\documentclass[11pt]{amsart}
\usepackage[a4paper,margin=28mm]{geometry}
\usepackage[T1]{fontenc}
\usepackage{lmodern}
\usepackage{microtype}
\usepackage{amsmath,amssymb,amsthm,mathtools}
\usepackage{enumitem}
\usepackage{xcolor}
\usepackage{hyperref}
\hypersetup{colorlinks=true,linkcolor=blue!50!black,citecolor=blue!50!black,urlcolor=blue!50!black}
\allowdisplaybreaks
\numberwithin{equation}{section}
\newtheorem{theorem}{Theorem}[section]
\newtheorem{proposition}[theorem]{Proposition}
\newtheorem{lemma}[theorem]{Lemma}
\newtheorem{corollary}[theorem]{Corollary}
\theoremstyle{definition}

\newtheorem{question}[theorem]{Question}
\theoremstyle{remark}
\newtheorem{remark}[theorem]{Remark}
\newcommand{\R}{\mathbb R}
\newcommand{\C}{\mathbb C}
\newcommand{\Z}{\mathbb Z}
\newcommand{\N}{\mathbb N}
\newcommand{\T}{\mathbb T}
\newcommand{\1}{\mathbf 1}

\newcommand{\diag}{\operatorname{diag}}

\newcommand{\dd}{\,\mathrm d}
\newcommand{\norm}[1]{\left\lVert #1\right\rVert}

\newcommand{\inner}[2]{\left\langle #1,#2\right\rangle}

\begin{document}
\title[Calder\'on commutators on quantum spaces]{Calder\'on Commutators on Quantum Euclidean Spaces and Quantum Tori}
\author{Deyu Chen}
\address{Deyu Chen, Institute for Advanced Study in Mathematics,
Harbin Institute of Technology, Harbin 150001, China}
\email{1201200317@stu.hit.edu.cn}

\author{Guixiang Hong}
\address{Guixiang Hong, Institute for Advanced Study in Mathematics,
Harbin Institute of Technology, Harbin 150001, China}
\email{gxhong@hit.edu.cn}

\date{}
\date{\today}

\begin{abstract}
In the classical setting, Calder\'on's first commutator theorem states
that Lipschitz regularity of the symbol characterizes the
$L_2$-boundedness of the first commutators. We show that the formally analogous assertion fails on every quantum
Euclidean space $\R^d_\theta$ with $\theta\neq 0$.
The same simultaneous failure occurs on $\T_\theta^d$ whenever
$\theta$ has an irrational entry. For rational $\theta$, we obtain the optimal logarithmic bound for the corresponding
Calder\'on first commutator. We also discuss the connection between Calder\'on's first commutator
and the $T1$ theorem, and construct an operator-valued counterexample to Calder\'on's
first commutator theorem, which shows that there is no direct analogue
of the classical $T1$ theorem in the operator-valued setting.  

The counterexamples and the optimality argument rely on the construction of sequences of matrices $\mathsf D_n$ and $\mathsf A_n$ such that $\|[\mathsf D_n,\mathsf A_n]\|$ is uniformly bounded, whereas $\|[|\mathsf D_n|,\mathsf A_n]\|$ tends to infinity. This phenomenon reflects the failure of the
absolute value function to be operator Lipschitz with
respect to the operator norm. The logarithmic estimate,
on the other hand, is obtained by combining matrix-valued
$T1$ theory with a transference argument.
\end{abstract}

\maketitle
\enlargethispage{3pt}

\section{Introduction}\label{sec:introduction}
\subsection{Background}
\subsubsection{Commutator theory in classical and quantum settings}
Calder\'on's first commutator theorem is a basic endpoint regularity result
in classical harmonic analysis. It asserts that if $b$ is a scalar Lipschitz
function on $\R^d$, then
\begin{equation}\label{eq:classical-estimate}
 \|[b,\partial_jR_j]\|_{B(L_2(\R^d))}
 \leq C_d\|\nabla b\|_\infty,
\end{equation}
where $R_j$ denotes the $j$th Riesz transform.
This estimate originates in Calder\'on's work \cite{Calderon1965}; see also
\cite{CoifmanMeyer1975,Muscalu2014}. The converse was proved by Murray in
$d=1$ \cite{Murray1985} and later, in general $d\ge 1$, by Chen, Ding,
and Hong \cite{ChenDingHong2016}. Consequently, in the classical
Euclidean setting, Lipschitz regularity characterizes the boundedness
of Calder\'on's first commutator.
This should be contrasted with the commutator with the zero-order Riesz
transform $R_j$. By the classical Coifman--Rochberg--Weiss theory \cite{CoifmanRochbergWeiss1976},
\begin{equation}\label{eq:classical-estimate2}
 \|[b,R_j]\|_{B(L_2(\R^d))}
 \leq C_d\|b\|_{\operatorname{BMO}(\R^d)}.
\end{equation}
Thus the two commutators lie in different regularity regimes.

In recent years, the development of noncommutative harmonic analysis has
led to growing interest in harmonic analysis on quantum Euclidean spaces
and quantum tori. These algebras provide natural models in which
the Fourier variables and differential calculus remain classical, while
multiplication is deformed. A substantial harmonic analysis theory has
been developed in these quantum settings, including Fourier multipliers,
Sobolev and Besov spaces, and singular integral operators; see, for instance,
\cite{ChenXuYin2013,XiongXuYin2018,GJP,GJM,HLW2023,hlw25,ChenHong2026}.
In particular, commutator estimates and quantum differentiability were
studied by McDonald, Sukochev, and Xiong, who established a quantum
differentiability theorem on quantum tori
\cite{MSX2019}. For $d\geq2$ and a bounded symbol
$b\in L_\infty(\T_\theta^d)$, their result gives the characterization
\begin{equation}\label{eq:quantum-torus-Schatten}
 [b,R_j]\in S_{d,\infty}(L_2(\T_\theta^d))
 \quad\text{for all }1\leq j\leq d
 \quad\Longleftrightarrow\quad
 b\in\dot W^{1,d}(\T_\theta^d),
\end{equation}
where $\dot W^{1,d}(\T_\theta^d)$ denotes the homogeneous Sobolev
space on $\T_\theta^d$ and $S_{d,\infty}$ denotes the weak
Schatten-$d$ class. They \cite{MSX2020,MSX2023} later obtained an analogous
characterization on quantum Euclidean spaces $\mathbb R_\theta^d$; see also
\cite{SukochevXiongZanin2023,Tian2025} for related developments.

In contrast, fewer results are known about Calder\'on's first commutator
in the quantum setting. Compared with the Riesz commutator $[b,R_j]$,
the operator $[b,\partial_jR_j]$ is generally
noncompact, and lies at a more delicate boundedness endpoint, where
stronger regularity of the symbol $b$ is required. The noncommutativity of the underlying multiplication
also prevents a direct application of classical tools such as
Lusin area-function estimates and the $T1$ theorem. Moreover, unlike
Schatten-class problems, the endpoint problem in $B(L_2)$ does not
come with compactness or singular-value information, so spectral
methods and operator-integral techniques that are effective in the
Schatten setting are no longer directly available. These features make
the quantum first Calder\'on commutator substantially more delicate.

\subsubsection{Classical and operator-valued T1 theory}
The $T1$ theorem of David and Journ\'e~\cite{DavidJourne1984}
gives an $L_2$-boundedness criterion for standard Calder\'on--Zygmund
operators. Let $T$ be such an operator. If $T$ satisfies the BMO
testing conditions
\begin{equation}\label{eq:BMO}
 T1,T^*1\in\operatorname{BMO}(\R^d),
\end{equation}
together with the weak boundedness property
\begin{equation}\label{eq:weak bound}
 \bigl|\langle T\phi_Q,\psi_Q\rangle\bigr|
 \lesssim |Q|,
\end{equation}
for every cube $Q\subset\R^d$ and every
$\phi_Q,\psi_Q\in C_c^\infty(\R^d)$ supported in $Q$ and satisfying
the standard normalization
\[
 \|\phi_Q\|_\infty,\|\psi_Q\|_\infty\leq1,\qquad
 \|\nabla\phi_Q\|_\infty,\|\nabla\psi_Q\|_\infty
 \lesssim |Q|^{-1/d},
\]
then $T$ extends to a bounded operator on $L_2(\R^d)$. Standard
Calder\'on--Zygmund theory then yields boundedness on $L_p(\R^d)$ for
every $1<p<\infty$.

Calder\'on's first commutator theorem may also be recovered from the
$T1$ theorem. Indeed, the Lipschitz regularity of the symbol $b$ yields
the standard kernel estimates for the commutator
$[b,\partial_jR_j]$, while its special commutator structure provides
the required BMO testing and weak boundedness conditions. Thus the
$T1$ theorem places Calder\'on's first commutator theorem within a
broader Calder\'on--Zygmund framework.

The most natural operator-valued analogue
of the scalar $T1$ conditions is obtained by replacing the absolute
values in the kernel size and smoothness estimates, as well as in the
weak boundedness property, by operator norms, and by replacing scalar
BMO with an operator-valued BMO space; see, for instance,
\cite{HongLiuMei2020}.
Nazarov, Pisier, Treil, and
Volberg~\cite{NazarovPisierTreilVolberg2002} established the sharp
logarithmic dependence on the matrix dimension for matrix-valued
Carleson embeddings and paraproducts. For matrix-weighted $A_2$ criteria and related paraproduct estimates,
see \cite{Isralowitz2018,PottStoica2017,IsralowitzKwonPott2017}.

The dependence on the matrix dimension cannot, in general, be removed.
Indeed, Mei~\cite{Mei2006} showed that even the $L_\infty$ norm of a
matrix-valued paraproduct symbol does not control the paraproduct norm
uniformly in the matrix dimension. Consequently, the natural operator-valued analogues of the BMO testing \eqref{eq:BMO} and
weak boundedness conditions \eqref{eq:weak bound} from the scalar theory are known to be
insufficient in general for $L_2$-boundedness. This reveals a fundamental obstruction to obtaining an unrestricted
$T1$ theorem by simply replacing scalar BMO with an operator-valued
BMO space, particularly in the infinite-dimensional setting. 
Under additional structural assumptions, however, Hong,
Liu, and Mei~\cite{HongLiuMei2020} established an $L_2$-boundedness
criterion for Calder\'on--Zygmund operators with
$B(\ell_2)$-valued kernels.
Assuming the uniform operator-norm size and regularity estimates in~\cite[(1.5), (1.6)]{HongLiuMei2020} for the operator-valued
Calder\'on--Zygmund kernels, their theorem combines
operator-valued BMO testing with the symmetry condition
\(
 (T1)^*=T^*1,
\)
and therefore requires additional structure beyond the classical
scalar $T1$ theorem.
Operator-valued $T1$ theory on UMD spaces was developed by
Hyt\"onen et al.~\cite{HytonenWeis2006,HanninenHytonen2016,Hytonen2021}.
A key additional assumption in these results is the
Rademacher-boundedness of the relevant kernel families, which is
strictly stronger in general than uniform boundedness in operator norm in \cite{HongLiuMei2020}.

Gonz\'alez-P\'erez, Junge, and Parcet~\cite{GJP} established a
Calder\'on--Zygmund theory on quantum Euclidean spaces. This naturally
leads to the question of a corresponding quantum $T1$ theorem.
The quantum problem presents an additional difficulty beyond the
operator-valued setting: the latter retains the classical spatial
variables and local bump-function tests, whereas a suitable intrinsic
formulation of the weak boundedness property on quantum Euclidean
spaces is not evident. From the perspective of $T1$ theory, it is
therefore of particular interest to determine whether Calder\'on's
first commutator theorem remains valid in the %operator-valued and
quantum settings. Indeed, it is listed as a conjecture in the survey \cite[Conjecture 8.16]{HWX}.

\subsection{Main results}

The main purpose of this paper is to determine whether the Lipschitz
characterization of the boundedness of Calder\'on's first commutator
persists in the quantum setting.
We first introduce the notation needed to state the main results. Let
$d\geq2$ and let $\theta$ be a real antisymmetric $d\times d$ matrix.
We write $\R^d_\theta$ for the quantum Euclidean space associated with a
real antisymmetric matrix $\theta$, namely, the von Neumann algebra
generated by the Weyl unitaries $\{\lambda_\theta(\xi):\xi\in\R^d\}$
satisfying
\begin{equation}\label{eq:intro-weyl}
 \lambda_\theta(\xi)\lambda_\theta(\eta)
 =e^{\mathrm i\xi^{\mathsf T}\theta\eta/2}
   \lambda_\theta(\xi+\eta),\qquad \xi,\eta\in\R^d.
\end{equation}
Its canonical trace is determined by
\[
\tau_\theta(\lambda_\theta(f))=f(0),
\qquad
\lambda_\theta(f)=\int_{\R^d} f(\xi)\lambda_\theta(\xi)\,d\xi,
\]
for $f\in\mathcal S(\R^d)$. This trace gives rise to the
associated noncommutative $L_p$-spaces $L_p(\R_\theta^d)$,
$1\le p<\infty$, while $L_\infty(\R_\theta^d)$ is identified with
$\R_\theta^d$ itself. The derivations on $\mathbb R_\theta^d$ are determined by
$\partial_k\lambda_\theta(\xi)=\mathrm i\xi_k\lambda_\theta(\xi)$.
We define the Lipschitz class on $\mathbb R_\theta^d$ as the Banach space
\begin{equation}\label{eq:W1infty-definition}
 W^{1,\infty}(\mathbb R_\theta^d)
 :=\{b\in L_\infty(\mathbb R_\theta^d):\partial_kb\in L_\infty(\mathbb R_\theta^d),\ 1\leq k\leq d\}
\end{equation}
equipped with the norm $$ \|b\|_{W^{1,\infty}}:=\|b\|_\infty+\sum_{k=1}^d\|\partial_kb\|_\infty.$$
Derivatives of bounded elements are understood distributionally. Put
\begin{equation}\label{eq:intro-multipliers}
 L:=-\sum_{k=1}^d\partial_k^2,\qquad
 R_j:=\partial_jL^{-1/2},\qquad
 {T_j:=-\partial_jR_j.}
\end{equation}
 On quantum Euclidean
space, $T_j$ has Fourier symbol
\begin{equation}\label{eq:classical-symbol}
 m_j(\xi)=\frac{\xi_j^2}{|\xi|}\quad(\xi\ne0),\qquad m_j(0)=0.
\end{equation}

\begin{theorem}\label{thm:euclidean}
Let $d\geq2$ and $j\in\{1,\ldots,d\}$. The following assertions are equivalent:
\begin{enumerate}
\item[{\rm (i)}] $\theta=0$;
\item[{\rm (ii)}] for every $b\in W^{1,\infty}(\R^d_\theta)$, the commutator
$[b,\partial_jR_j]$ extends to an element of $B(L_2(\R^d_\theta))$.
\end{enumerate}
Moreover, if $\theta\ne0$, there is a self-adjoint
$b\in W^{1,\infty}(\R^d_\theta)$ for which
$[b,\partial_kR_k]$ has no bounded extension for every $1\le k\le d$.
\end{theorem}

We also establish analogous results on quantum tori.
The quantum torus $\T_\theta^d$ is
the von Neumann algebra generated by unitaries $U_1,\ldots,U_d$ satisfying
\begin{equation}\label{eq:intro-torus}
 U_kU_l=e^{2\pi\mathrm i\theta_{kl}}U_lU_k.
\end{equation}
Its canonical trace is determined by $$\tau(U^n)=0,\qquad 0\ne n\in \mathbb Z^d,\qquad \tau(1)=1,$$ where $U^n:=U_1^{n_1}\cdots U_d^{n_d}$. This trace gives rise
to the associated noncommutative $L_p$-spaces $L_p(\T_\theta^d)$,
$1\le p<\infty$, while $L_\infty(\T_\theta^d)$ is identified with
$\T_\theta^d$ itself.
The derivations on $\mathbb T_\theta^d$ are determined by $\partial_kU^n=\mathrm i n_kU^n$, so that $T_j:=-\partial_jR_j$ has Fourier symbol $$m_j(n)=\frac{n_j^2}{|n|}\quad(n\ne0),\qquad m_j(0)=0.$$ 
We define the Lipschitz class on $\mathbb T_\theta^d$ as the Banach space
\begin{equation}\label{eq:W1infty-definition2}
 W^{1,\infty}(\mathbb T_\theta^d)
 :=\{b\in L_\infty(\mathbb T_\theta^d):\partial_kb\in L_\infty(\mathbb T_\theta^d),\ 1\leq k\leq d\}
\end{equation}
equipped with the norm $$ \|b\|_{W^{1,\infty}}:=\|b\|_\infty+\sum_{k=1}^d\|\partial_kb\|_\infty.$$

\begin{theorem}\label{thm:torus}
Let $d\geq2$ and suppose that $\theta$ has at least one irrational entry.
There exists a self-adjoint $b\in W^{1,\infty}(\T_\theta^d)$ such that
$[b,\partial_jR_j]$ has no bounded extension to $L_2( \T_\theta^d)$
for every $1\leq j\leq d$.
\end{theorem}

Theorem \ref{thm:torus} does not yield an
analogous conclusion for rational deformation matrices. In contrast,
we prove that Calder\'on's first commutator theorem remains valid for
every rational deformation matrix $\theta$, with optimal logarithmic
dependence on the common denominator.

\begin{theorem}\label{thm:torus-rational}
Let $d\geq2$, $\theta$ be a rational antisymmetric $d\times d$
matrix and let $q\geq1$ be an integer such that $q\theta$ has integer
entries. For every $b\in W^{1,\infty}(\T_\theta^d)$, the commutator
$[b,\partial_jR_j]$ has a bounded extension to $L_2(\T_\theta^d)$, and
\begin{equation}\label{eq:rational-torus-bound}
 \|[b,\partial_jR_j]\|_{B(L_2(\T_\theta^d))}
 \leq C_d\log(1+q)\sum_{k=1}^d\|\partial_kb\|_\infty,
 \qquad 1\leq j\leq d.
\end{equation}
Here $C_d$ depends only on $d$. Moreover, the logarithmic dependence on $q$ is optimal.
\end{theorem}

% Theorem~\ref{thm:torus} also follows from the sharpness assertion
% in Theorem~\ref{thm:torus-rational} by an approximation argument
% given at the end of Section~\ref{subsec:rational-tori}.

Finally, as a byproduct, we show that the operator-valued
analogue of Calder\'on's first commutator theorem fails
in general. Our construction also provides another counterexample to the sufficiency of naive operator-valued analogues of the classical $T1$ conditions for $L_2$-boundedness, distinct from the paraproduct counterexample
in~\cite{Mei2006}; see also~\cite{Hytonen2021}.

\begin{theorem}\label{thm:operator-valued-Calderon}
There exists a self-adjoint function
$b\in W^{1,\infty}(\R;B(\ell_2))$ such that
$\left[b,\left|\frac{\mathrm d}{\mathrm dx}\right|\right]$
has no bounded extension to $L_2(\R;\ell_2)$.
\end{theorem}

The construction of counterexamples in Theorems~\ref{thm:euclidean}, \ref{thm:torus}, and \ref{thm:operator-valued-Calderon} comes from the failure of the scalar
Lipschitz function $t\mapsto|t|$ to be commutator Lipschitz in operator
norm. In the infinite dimensional setting, McIntosh \cite{McIntosh1971} constructed self-adjoint operators
$D$ and $A$ for which $[D,A]$ is bounded but $[|D|,A]$ is unbounded.
Related operator-norm continuity
estimates for the absolute-value map were obtained by Kato
\cite{Kato1973}; see also \cite{AleksandrovPeller2016} for a broader
discussion of operator-Lipschitz functions.
In the finite-dimensional setting, Mathias \cite{Mathias1993} studied
Schur multipliers associated with the symbol
\[
 \left(\frac{\lambda_r-\lambda_s}{\lambda_r+\lambda_s}\right)_{r,s},
 \qquad \lambda_r>0,
\]
from which one obtains the existence of $2N\times2N$ matrices $D_N$ and $E_N$
satisfying
\begin{equation}\label{eq:intro-matrix-obstruction}
 \|[D_N,E_N]\|\leq\pi,\qquad
 \|[|D_N|,E_N]\|\gtrsim\log N.
\end{equation}
By a suitable rescaling and choice of dimensions $N_n$, one may therefore
obtain matrices $\mathsf D_n$ and $\mathsf A_n$ such that
\begin{equation}\label{eq:intro-matrix-obstruction1}
 \|[\mathsf D_n,\mathsf A_n]\|\leq 2^{-n},
 \qquad
 \|[|\mathsf D_n|,\mathsf A_n]\|\gtrsim n.
\end{equation}
In Section~\ref{s2}, we give explicit constructions of such
matrices, with a specially chosen spectrum that is useful for the
subsequent construction of counterexamples.

The main strategy for the proof of Theorem \ref{thm:euclidean} and \ref{thm:torus} is as follows: We first construct the counterexamples in the two-dimensional case,
where $\R_\theta^2$, $\theta\ne0$, is isomorphic to the Moyal plane
$B(L_2(\R))$. The basic idea is to realize the finite-dimensional
matrix obstruction on mutually orthogonal spatial blocks. In the
Schr\"odinger model, we choose compactly supported packets with pairwise
disjoint supports and associate to each pair $(\mathsf D_n,\mathsf A_n)$
a symbol $\mathfrak b_n$ supported on the corresponding block. The construction is
arranged so that
\[
 \|\partial_1\mathfrak b_n\|+\|\partial_2\mathfrak b_n\|
 \lesssim 2^{-n},
\]
whereas the test isometries $J_n$ constructed in
Section~\ref{s3.3} give
\[
 \|J_n^*[\mathfrak b_n,T_2]J_n\|
 \geq \|[|\mathsf D_n|,\mathsf A_n]\|-C2^{-n}
 \geq n-C2^{-n}.
\]
Thus we may sum the blocks to obtain
\(
 \mathfrak b=\sum_{n\ge1} \mathfrak b_n
\)
with bounded derivatives and an unbounded commutator
$[\mathfrak b,T_2]$. The general quantum Euclidean case follows from the tensor
decomposition. If $\operatorname{rank}\theta=2r>0$, then, after a change of variables,
\[
 Q^{\mathsf T}\theta Q
 =
 \theta_1J\oplus\cdots\oplus\theta_rJ\oplus0_{d-2r},
 \qquad
 J=
 \begin{pmatrix}
 0&1\\
 -1&0
 \end{pmatrix},
\]
with $\theta_k\ne0$, and consequently
\[
 \R_\theta^d
 \simeq
 \overline\bigotimes_{k=1}^rB(L_2(\R))
 \,\overline\otimes\,
 L_\infty(\R^{d-2r}).
\]
Thus the two-dimensional matrix obstruction can be embedded into every
Moyal factor, which combining with a change of variable gives the construction of desired symbol $b$ on $\R_\theta^d$.

The construction on noncommutative tori follows the same principle,
but the finite-dimensional obstruction is realized through orbit towers.
Writing the irrational rotation algebra as
$L_\infty(\mathbb T)\rtimes_{\alpha_\theta}\mathbb Z$, with the action
\(
 \alpha_\theta(z)=e^{2\pi i\theta}z,
\)
the irrationality of $\theta$ implies that, for every $N\ge1$, the points
\[
 z,\alpha_\theta(z),\ldots,\alpha_\theta^{N-1}(z)
\]
are distinct. Hence one may choose a sufficiently small interval
$I\subset\mathbb T$ such that
\[
 I,\alpha_\theta(I),\ldots,\alpha_\theta^{N-1}(I)
\]
are pairwise disjoint. Smooth functions supported on these translates
then form an $N$-level orbit tower, on which the relevant operators
realize the prescribed finite-dimensional matrix action.   Once such orbit towers are available, the same comparison between first
derivatives and $[\mathsf D_n,\mathsf A_n]$, and between the Calder\'on
commutator and $[|\mathsf D_n|,\mathsf A_n]$, yields the desired
counterexample.
 For quantum tori, the corresponding
reduction is obtained by embedding an irrational rotation algebra along
a two-dimensional integer sublattice.

For rational deformation matrices, we prove
Theorem~\ref{thm:torus-rational} by a finite-dimensional matrix
realization. If \(q\theta\) has integer entries, we transfer
\(\T_\theta^d\) to
\(L_\infty(\T^d)\,\overline\otimes\,\mathbb M_{q^d}(\C)\)
via a trace-preserving \(*\)-homomorphism that intertwines the derivations and Fourier multipliers in the two settings. This transference argument is standard in harmonic analysis on quantum settings; see, for instance, \cite{XiongXuYin2018,GJP}.
The dyadic representation theorem, the uniform bound of matrix-valued dyadic shifts and the logarithmic bound for
matrix-valued paraproducts give a Calder\'on commutator estimate
with logarithmic dependence on the matrix dimension. Transferring
this estimate to the quantum torus yields the upper bound
$C_d\log(1+q)$. To prove its optimality, we use a $q$-dimensional realization of the two-dimensional rational rotation algebra with parameter $1/q$.
Localized matrix functions are extended by its finite translation action so that they satisfy the exact Fourier constraint of the quantum torus. High-frequency modulation then recovers the
finite-matrix obstruction from Section~\ref{s2} and gives the
matching logarithmic lower bound.

Finally, we emphasize that Calder\'on's first commutator theorem fits naturally into the classical $T1$ framework developed by David and Journ\'e. In the operator-valued
and quantum settings, however, no unrestricted analogue of the classical
$T1$ principle is available for this purpose. Theorem~\ref{thm:operator-valued-Calderon} provides an operator-valued
Calder\'on commutator that is a $T1$ operator but is not $L_2$-bounded,
even though both $T1$ and $T^*1$ are bounded. Our counterexamples on quantum Euclidean spaces reveal a parallel
obstruction: bounded first derivatives of the symbol are insufficient
to recover the classical Calder\'on boundedness principle. Thus, identifying the optimal $T1$ hypotheses in the operator-valued and quantum settings remains an open problem.

\subsection{Organization and Notation}
The paper is organized as follows. Section~\ref{s2} presents the finite-dimensional matrix construction. Section~\ref{sec:euclidean} proves Theorem~\ref{thm:euclidean}, first in dimension two and then in higher dimensions via a tensor reduction. Section~\ref{sec:torus} proves Theorem~\ref{thm:torus} by a similar reduction. Section~\ref{subsec:rational-tori} proves Theorem~\ref{thm:torus-rational}, including the optimal logarithmic dependence on the denominator $q$. Finally, Section~\ref{sec:t1} proves Theorem~\ref{thm:operator-valued-Calderon} and discusses the operator-valued and quantum $T1$ theory.

Throughout, we let $\T=:\R/(2\pi\Z)$ equipped with the normalized Haar measure. We use $\|\cdot\|$ for an operator norm and $\|\cdot\|_p$ for a
noncommutative $L_p$ norm. Constants denoted by $C$ may change from line
to line.  We write $S_2(H)$ and $K(H)$ for the Hilbert--Schmidt and compact
operators on a Hilbert space $H$, respectively. For a positive integer $N,$ let $\mathbb M_N(\C)$ denote the $N\times N$  complex matrix algebra and $\mathbf{Id}_{N}$ denote the identity matrix in $\mathbb M_N(\C)$.

\section{A finite-dimensional obstruction}\label{s2}

This section contains the matrix construction used in all three settings.
Its essential feature is that a small commutator with a diagonal matrix
need not imply a small commutator with its absolute value. The matrices
below give an explicit logarithmic separation.

For \(N\geq2\), set
\begin{equation}\label{eq:DN-definition}
  \Lambda_N:=\diag(2,2^2,\ldots,2^N),
  \qquad
  D_N:=
  \begin{pmatrix}
    \Lambda_N&0\\
    0&-\Lambda_N
  \end{pmatrix}.
\end{equation}
Let \(G_N=(g_{rs})_{r,s=1}^N\) be the real antisymmetric matrix defined by
\begin{equation}\label{eq:GN-definition}
  g_{rs}
  :=
  \begin{cases}
    \displaystyle
    \frac{1}{(2^r+2^s)(r-s)},&r\neq s,\\[6pt]
    0,&r=s,
  \end{cases}
\end{equation}
and define the self-adjoint \(2N\times2N\) matrix
\begin{equation}\label{eq:EN-definition}
  E_N:=
  \begin{pmatrix}
    0&G_N\\
    G_N^*&0
  \end{pmatrix}.
\end{equation}

\begin{proposition}\label{prop:finite-obstruction}
For every \(N\geq2\),
\begin{align}
  \norm{E_N}&\leq\frac{\pi}{6},\label{eq:EN-bound}\\
  \norm{[D_N,E_N]}&\leq\pi,\label{eq:DN-EN-bound}\\
  \norm{[|D_N|,E_N]}
  &\geq
  \kappa_N
  :=
  \frac{2}{3}
  \left(
   \sum_{r=1}^{N-1}r^{-1}-\frac{N-1}{N}
  \right).
  \label{eq:modulus-bound}
\end{align}
\end{proposition}

\begin{proof}
The Hilbert--Schmidt norm gives
\begin{align*}
  \norm{G_N}^2
  &\leq \norm{G_N}_{S_2}^2
    =\sum_{r\neq s}
      \frac{1}{(2^r+2^s)^2(r-s)^2}\\
  &\leq
    2\sum_{r=2}^\infty 4^{-r}
      \sum_{s=1}^{r-1}\frac{1}{(r-s)^2}\\
  &\leq
    2\left(\sum_{r=2}^\infty4^{-r}\right)
      \left(\sum_{k=1}^\infty\frac{1}{k^2}\right)
   =
    2\cdot\frac{1}{12}\cdot\frac{\pi^2}{6}
   =
    \frac{\pi^2}{36},
\end{align*}
which proves \eqref{eq:EN-bound} since $\|E_N\|=\|G_N\|$.
Let \(X_N:=(x_{rs})_{r,s=1}^N\), where \(x_{rs}=(r-s)^{-1}\) for \(r\neq s\)
and \(x_{rr}=0\). Then it follows that
\begin{equation}\label{XN}
    X_N=\Lambda_NG_N+G_N\Lambda_N.
\end{equation}
The convolution operator $X=(x_{rs})_{r,s\in\Z}$ on
$\ell_2(\Z)$ with kernel $1/k$, $k\ne0$, is unitarily equivalent
to multiplication on $L_2(\T)$ by the bounded function
\[
  t\mapsto \sum_{k\neq0}\frac{e^{-\mathrm ikt}}{k}
  =
  \mathrm i(t-\pi),
  \qquad 0<t<2\pi.
\]
Its operator norm is therefore \(\pi\). Since $X_N$ is a finite
compression of $X$, it follows that \(\norm{X_N}\leq\pi\).
From \eqref{XN}, a direct calculation yields
\[
  [D_N,E_N]
  =
  \begin{pmatrix}
    0&X_N\\
    -X_N^*&0
  \end{pmatrix},
\]
and hence \eqref{eq:DN-EN-bound}.

Since \(|D_N|=\diag(\Lambda_N,\Lambda_N)\),
\[
  [|D_N|,E_N]
  =
  \begin{pmatrix}
    0&Y_N\\
    -Y_N^*&0
  \end{pmatrix},
  \qquad
  Y_N=\Lambda_NG_N-G_N\Lambda_N.
\]
For \(r\neq s\),
\begin{equation}\label{eq:YN-entry}
  (Y_N)_{rs}
  =
  \frac{2^r-2^s}{2^r+2^s}\frac{1}{r-s}
  =
  \frac{2^{|r-s|}-1}{2^{|r-s|}+1}
  \frac{1}{|r-s|}.
\end{equation}
Thus \(Y_N\) is real symmetric and has nonnegative off-diagonal entries. Let
\[
 \xi_N=N^{-1/2}(1,\ldots,1)^{\mathsf T}.
\]
Since $(2^k-1)/(2^k+1)\geq1/3$ for $k\geq1$, we obtain
\begin{align*}
  \norm{Y_N}
  &\geq \inner{Y_N\xi_N}{\xi_N}\\
  &=
    \frac{2}{N}
    \sum_{k=1}^{N-1}
    (N-k)\frac{2^k-1}{2^k+1}\frac{1}{k}\\
  &\geq
    \frac{2}{3N}
    \sum_{k=1}^{N-1}\frac{N-k}{k}
   =
    \frac{2}{3}
    \left(
     \sum_{r=1}^{N-1}r^{-1}-\frac{N-1}{N}
    \right).
\end{align*}
Finally, the norm of the commutator $[|D_N|,E_N]$ equals \(\norm{Y_N}\), proving
\eqref{eq:modulus-bound}.
\end{proof}

\begin{corollary}\label{cor:scaled-obstruction}
For every \(n\geq1\), there exist a diagonal self-adjoint matrix \(\mathsf{D}_n\)
with distinct nonzero eigenvalues and a self-adjoint matrix \(\mathsf{A}_n\),
with $\|\mathsf A_n\|\leq 2^{-n}/6$, such that
\begin{equation}\label{eq:scaled-obstruction}
  \norm{[\mathsf{D}_n,\mathsf{A}_n]}\leq2^{-n},
  \qquad
  \norm{[|\mathsf{D}_n|,\mathsf{A}_n]}\geq n.
\end{equation}
\end{corollary}

\begin{proof}
Choose \(N_n\) so large that
\[\kappa_{N_n}=\frac{2}{3}
  \left(
   \sum_{r=1}^{N_n-1}r^{-1}-\frac{N_n-1}{N_n}
  \right)\geq\pi n2^n,\] and set
\begin{equation}\label{eq:Dn-An-definition}
  \mathsf{D}_n:=D_{N_n},
  \qquad
  \mathsf{A}_n:=\frac{2^{-n}}{\pi}E_{N_n}.
\end{equation}
Proposition~\ref{prop:finite-obstruction} gives
\eqref{eq:scaled-obstruction} and the stated bound for $\|\mathsf A_n\|$. The eigenvalues of \(\mathsf{D}_n\) are
\(\pm2^r\), \(1\leq r\leq N_n\), and are therefore distinct and nonzero.
\end{proof}

\section{Proof of Theorem~\ref{thm:euclidean}}\label{sec:euclidean}

In this section, we prove Theorem~\ref{thm:euclidean}. We first recall the differential structure of quantum Euclidean spaces and identify the Moyal plane with $B(L_2(\R))$ through the Schr\"odinger representation. Using spatially separated packets, we realize the finite-dimensional obstruction from Section~\ref{s2} to construct a self-adjoint Lipschitz symbol whose Calder\'on's first commutator is unbounded. We then extend the construction to every
nonzero deformation matrix by
the tensor decomposition of quantum Euclidean spaces and a suitable linear change of variables.

\subsection{Quantum Euclidean space and its differential structure}

We recall the definitions needed below, following
\cite{ChenHong2026,GJP,MSX2020} and retaining
the conventions fixed in the introduction. Let $\theta$ be a real
antisymmetric $d\times d$ matrix. The quantum Euclidean space
$\R_\theta^d$ is the von Neumann algebra generated by the twisted
left regular representation on $L_2(\R^d)$,
\[
 (\lambda_\theta(\xi)f)(x)
 :=e^{\mathrm i\xi^{\mathsf T}\theta x/2}f(x-\xi),
 \qquad \xi,x\in\R^d.
\]
The generating unitaries satisfy
\begin{equation}\label{eq:general-weyl-relations}
 \lambda_\theta(\xi)\lambda_\theta(\eta)
 =e^{\mathrm i\xi^{\mathsf T}\theta\eta/2}
       \lambda_\theta(\xi+\eta),
 \qquad
 \lambda_\theta(\xi)^*=\lambda_\theta(-\xi).
\end{equation}
For $f\in L_1(\R^d)$, we define
\[
 \lambda_\theta(f)
 =\int_{\R^d}f(\xi)\lambda_\theta(\xi)\dd\xi,
 \qquad
 \|\lambda_\theta(f)\|_\infty\leq\|f\|_1,
\]
where the integral is understood in the strong operator sense. The
quantum Schwartz class is defined by
\begin{equation}\label{eq:qes-schwartz}
 \mathcal S(\R_\theta^d)
 :=\{\lambda_\theta(f):f\in\mathcal S(\R^d)\}.
\end{equation}
We equip $\mathcal S(\R_\theta^d)$ with the Fr\'echet topology transported
from the usual Schwartz topology through $\lambda_\theta$. Its continuous
linear dual is denoted by $\mathcal S'(\R_\theta^d)$ and is referred to
as the space of tempered distributions on $\R_\theta^d$. For further
details, see, for instance,
\cite{GJP,MSX2020}.

The canonical normal faithful semifinite trace $\tau_\theta$ on
$\R_\theta^d$ is determined on the quantum Schwartz class by
\[
 \tau_\theta(\lambda_\theta(f))=f(0),
 \qquad f\in\mathcal S(\R^d).
\]
For $f,g\in\mathcal S(\R^d)$, the Weyl relations yield
\begin{equation}\label{eq:qes-trace-pairing}
 \begin{aligned}
  \lambda_\theta(f)^*
    &=\lambda_\theta\bigl(\overline{f(-\,\cdot)}\bigr),\\
  \tau_\theta\bigl(\lambda_\theta(f)\lambda_\theta(g)\bigr)
    &=\int_{\R^d}f(\xi)g(-\xi)\dd\xi.
 \end{aligned}
\end{equation}
For $1\leq p<\infty$, the associated noncommutative $L_p$-space
$L_p(\R_\theta^d)$ consists of the $\tau_\theta$-measurable operators
affiliated with $\R_\theta^d$ such that
\[
 \|x\|_p
 :=\bigl(\tau_\theta(|x|^p)\bigr)^{1/p}<\infty.
\]
We set $L_\infty(\R_\theta^d):=\R_\theta^d$, equipped with the operator
norm. It was proved in \cite{GJP} that $\mathcal S(\R_\theta^d)$ is
norm dense in $L_p(\R_\theta^d)$ for $1\leq p<\infty$ and weak-$*$
dense in $L_\infty(\R_\theta^d)$.
In particular, \eqref{eq:qes-trace-pairing} gives the Plancherel identity
\begin{equation}\label{eq:qes-plancherel}
 \|\lambda_\theta(f)\|_2=\|f\|_{L_2(\R^d)},
 \qquad f\in\mathcal S(\R^d).
\end{equation}
Hence, by density, $\lambda_\theta$ extends uniquely to a unitary
operator
\[
 \lambda_\theta:L_2(\R^d)\longrightarrow L_2(\R_\theta^d),
\]
called the Weyl--Plancherel transform.
Moreover, by duality, $\lambda_\theta$ also extends to a topological
isomorphism
\[
 \lambda_\theta:\mathcal S'(\R^d)\longrightarrow
 \mathcal S'(\R_\theta^d),
\]
determined by
\begin{equation}\label{eq:qes-tempered-quantization}
 \bigl(\lambda_\theta(u),\lambda_\theta(f)\bigr)
 =\bigl(u,\widetilde f\bigr),
 \qquad
 \widetilde f(\xi)=f(-\xi),
 \qquad f\in\mathcal S(\R^d).
\end{equation}

On $\mathcal S(\R_\theta^d)$, the canonical partial derivatives are
determined by
\begin{equation}\label{eq:qes-partial-derivatives}
 \partial_k\lambda_\theta(f)
   =\lambda_\theta(\mathrm i\xi_k f(\xi)),
 \qquad 1\leq k\leq d.
\end{equation}
For a multi-index $\beta\in\mathbb N^d$, we write
\[
 |\beta|_1:=\beta_1+\cdots+\beta_d,\qquad
 \partial^\beta
 :=\partial_1^{\beta_1}\cdots\partial_d^{\beta_d},
 \qquad \partial^0:=\mathrm{id}.
\]
These derivatives extend to tempered distributions by duality:
\begin{equation}\label{eq:qes-distributional-calculus}
 (\partial^\beta w,x)
 =(-1)^{|\beta|_1}(w,\partial^\beta x),
 \qquad
 w\in\mathcal S'(\R_\theta^d),\quad
 x\in\mathcal S(\R_\theta^d).
\end{equation}
Equivalently, for $u\in\mathcal S'(\R^d)$,
\[
 \partial^\beta\lambda_\theta(u)
   =\lambda_\theta\bigl((\mathrm i\xi)^\beta u\bigr).
\]

For an integer $r\geq0$ and $1\leq p\leq\infty$, the Sobolev space
$W^{r,p}(\R_\theta^d)$ and its norm are defined by
\begin{equation}\label{eq:qes-sobolev}
 \begin{aligned}
 W^{r,p}(\R_\theta^d)
   &:=\{x\in\mathcal S'(\R_\theta^d):
        \partial^\beta x\in L_p(\R_\theta^d)
        \text{ for }|\beta|_1\leq r\},\\
 \|x\|_{W^{r,p}}
   &:=\sum_{|\beta|_1\leq r}\|\partial^\beta x\|_p.
 \end{aligned}
\end{equation}
In particular, $W^{1,\infty}(\R_\theta^d)$ agrees with the Lipschitz
class defined in \eqref{eq:W1infty-definition}. Moreover, regarding
each $\partial_k$ as its closure on $L_2(\R_\theta^d)$,
\[
 W^{1,2}(\R_\theta^d)
 =\bigcap_{k=1}^d\operatorname{Dom}(\partial_k).
\]

We now introduce our Fourier multiplier convention. For a measurable
symbol $m$ of at most polynomial growth, define
\begin{equation}\label{eq:qes-fourier-multiplier}
 m(D)\lambda_\theta(f):=\lambda_\theta(mf),
 \qquad f\in\mathcal S(\R^d).
\end{equation}
In particular, the operator
\(
 T_j=-\partial_jR_j
\)
is the Fourier multiplier with symbol
\begin{equation}\label{eq:euclidean-mj}
 m_j(\xi)=\frac{\xi_j^2}{|\xi|}\quad(\xi\ne0),
 \qquad m_j(0)=0.
\end{equation}
Since
\[
 |m_j(\xi)|=\frac{\xi_j^2}{|\xi|}
 \leq |\xi_j|,
\]
the Plancherel identity yields
\begin{equation}\label{eq:Tj-W12}
 \|T_jx\|_2
 \leq \|\partial_jx\|_2,
 \qquad x\in W^{1,2}(\R_\theta^d).
\end{equation}

The following lemma shows that when $b\in W^{1,\infty}(\R_\theta^d)$, Calder{\'o}n's first commutator $[b,T_j]$ is well-defined on $W^{1,2}(\R_\theta^d),$  which is dense in
$L_2(\R_\theta^d)$ since
$\mathcal S(\R_\theta^d)\subset W^{1,2}(\R_\theta^d)$ and
$\mathcal S(\R_\theta^d)$ is dense in $L_2(\R_\theta^d)$. 
\begin{lemma}
For every $b\in W^{1,\infty}(\R_\theta^d)$ and $1\leq j\leq d$,
the commutator $[b,T_j]$ is a bounded map from
$W^{1,2}(\R_\theta^d)$ to $L_2(\R_\theta^d)$.
\end{lemma}
\begin{proof}
    Let
$b\in W^{1,\infty}(\R_\theta^d)$ and
$x\in W^{1,2}(\R_\theta^d)$. The Leibniz rule gives
\begin{equation*}
 \partial_k(bx)
 =(\partial_kb)x+b\partial_kx,
\end{equation*}
and hence
\[
 \|\partial_k(bx)\|_2
 \leq
 \|\partial_kb\|_\infty\|x\|_2
 +\|b\|_\infty\|\partial_kx\|_2.
\]
Thus $bx\in W^{1,2}(\R_\theta^d)$. Using
\eqref{eq:Tj-W12}, we obtain
\begin{align*}
 \|[b,T_j]x\|_2
 &\leq
 \|bT_jx\|_2+\|T_j(bx)\|_2\\
 &\leq
 \|b\|_\infty\|\partial_jx\|_2
 +\|\partial_j(bx)\|_2\\
 &\leq
 \|\partial_jb\|_\infty\|x\|_2
 +2\|b\|_\infty\|\partial_jx\|_2\\
 &\lesssim
 \|b\|_{W^{1,\infty}}\|x\|_{W^{1,2}}.
\end{align*}
Consequently, $[b,T_j]$ is well defined as an operator from
$W^{1,2}(\R_\theta^d)$ to $L_2(\R_\theta^d)$.
\end{proof}

\subsection{The Moyal plane and its Schr\"odinger realization}

We first work with the Heisenberg matrix
\[
 \theta=\begin{pmatrix}0&1\\-1&0\end{pmatrix}.
\]
In this case, the Stone--von Neumann theorem identifies
$\R_\theta^2$ with the Moyal plane $B(L_2(\R))$, also referred to as
quantum phase space; see
\cite[Section~2.2.2]{MSX2020}. Hence Calder\'on's
first commutator problem on $\R_\theta^2$ can be transferred to the
corresponding problem on $B(L_2(\R))$.

We specify this realization and its normalization below. Recall that the
position and momentum operators are the self-adjoint closures of the
operators defined on $\mathcal S(\R)$ by
\[
 Qf(x)=xf(x),\qquad Pf(x)=-\mathrm i f'(x).
\]
The operators $Q$ and $P$ satisfy, on $\mathcal S(\R)$, the Heisenberg commutation relation
\[
 [Q,P]=\mathrm i\1.
\]
Their Weyl unitaries are defined by
\begin{align*}
 W(s,t)&=e^{\mathrm i(sQ-tP)},\\
 (W(s,t)f)(x)&=e^{\mathrm is(x-t/2)}f(x-t),
 \qquad (s,t)\in\R^2.
\end{align*}
These unitaries satisfy the Weyl relations
\begin{equation}\label{Weyl relation}
 W(s,t)W(s',t')
 =e^{\mathrm i(st'-ts')/2}W(s+s',t+t')
\end{equation}
and generate $B(L_2(\R))$ as a von Neumann algebra. Hence the transform
\begin{equation}\label{eq:plane-schrodinger-isomorphism}
 \Phi:\R_\theta^2\longrightarrow B(L_2(\R)),
 \qquad
 \Phi(\lambda_\theta(s,t))=W(s,t)
\end{equation}
extends to a normal $*$-isomorphism. Under this identification, $\lambda_\theta$ becomes the usual Weyl
transform:
\begin{equation}\label{eq:plane-integrated-weyl}
 \Phi(\lambda_\theta(f))
 =\int_{\R^2}f(s,t)W(s,t)\dd s\dd t,
 \qquad f\in\mathcal S(\R^2).
\end{equation}

The isomorphism $\Phi$ also intertwines the canonical traces up to a
fixed normalization. Let $\operatorname{Tr}$ denote the canonical trace
on $B(L_2(\R))$. With the above conventions,
\begin{equation}\label{trace}
 \tau_\theta=\frac{1}{2\pi}\operatorname{Tr}\circ\Phi.
\end{equation}
Consequently, for $1\leq p<\infty$, $\Phi$ identifies
$L_p(\R_\theta^2)$ with the Schatten class $S_p(L_2(\R))$, and
\[
 \|\Phi(x)\|_{S_p}
 =(2\pi)^{1/p}\|x\|_p.
\]

We next identify the canonical derivations in the Schr\"odinger
representation. A direct calculation gives
\[
 [P,W(s,t)]=sW(s,t),\qquad
 [Q,W(s,t)]=tW(s,t).
\]
For a bounded operator $X\in B(L_2(\R))$, commutators with the unbounded
operators $P$ and $Q$ are initially understood in the quadratic-form
sense on $\mathcal S(\R)$, and are said to be bounded when the
corresponding forms extend to bounded operators on $L_2(\R)$.
Comparing the above identities with
\eqref{eq:qes-partial-derivatives}, we obtain
\[
 \Phi(\partial_1 x)=\mathrm i[P,\Phi(x)],
 \qquad
 \Phi(\partial_2 x)=\mathrm i[Q,\Phi(x)],
 \qquad x\in\mathcal S(\R_\theta^2).
\]
Thus, under the Schr\"odinger representation, the first derivation is
implemented by the momentum operator and the second by the position
operator. 

For $u,v\in L_2(\R)$, let $u\otimes v^*$ denote the rank-one operator
$h\mapsto\langle h,v\rangle u$. The Fourier--Wigner transform is defined
initially on rank-one operators by
\begin{equation}\label{eq:plane-weyl-plancherel}
 \mathcal W(u\otimes v^*)(s,t)
 =(2\pi)^{-1/2}\langle u,W(s,t)v\rangle.
\end{equation}
By the orthogonality relations for the Weyl representation,
$\mathcal W$ extends to a unitary
\[
 \mathcal W:S_2(L_2(\R))\longrightarrow L_2(\R^2).
\]
For a measurable symbol $m$ that is finite almost everywhere, let
$M_m$ denote multiplication by $m$ on $L_2(\R^2)$. The associated Fourier--Wigner
multiplier is defined by
\begin{equation}\label{eq:moyal-multiplier-wigner}
 m(D)_{\mathrm M}
 :=\mathcal W^{-1}M_m\mathcal W;
\end{equation}
see, for instance, \cite{Samuelsen2025}.
For $f\in\mathcal S(\R^2)$ such that $mf\in L_2(\R^2)$, this agrees with
the Fourier multiplier transported from $\R_\theta^2$:
\begin{equation}\label{eq:plane-multiplier-unitary}
 m(D)_{\mathrm M}\Phi(\lambda_\theta(f))
 =
 \Phi(\lambda_\theta(mf)).
\end{equation}
Indeed, the trace normalization \eqref{trace} gives
\[
 \mathcal W\Phi(\lambda_\theta(f))
 =(2\pi)^{1/2}f,
 \qquad f\in\mathcal S(\R^2).
\]
Therefore,
\[
 \mathcal W m(D)_{\mathrm M}\Phi(\lambda_\theta(f))
 =(2\pi)^{1/2}mf
 =\mathcal W\Phi(\lambda_\theta(mf)),
\]
and \eqref{eq:plane-multiplier-unitary} follows from the unitarity of
$\mathcal W$. 

The following lemma shows that the boundedness of Calder\'on's first
commutator on $\R_\theta^2$ is equivalent to that of its realization on
the Moyal plane $B(L_2(\R))$. Hence, in the sequel, we may work entirely in the
Schr\"odinger realization of the Moyal plane.

\begin{lemma}\label{lem:plane-commutator-equivalence}
Let $b\in W^{1,\infty}(\R_\theta^2)$ and set
$\mathfrak b=\Phi(b)$. For $j=1,2$, let
\[
 T_{j,\mathrm M}:=m_j(D)_{\mathrm M},
 \qquad
 m_j(\xi)=\frac{\xi_j^2}{|\xi|}
 \quad (\xi\ne0),
 \qquad
 m_j(0)=0.
\]
Then we have
\[
 [b,T_j]\in B(L_2(\R_\theta^2))
 \quad\Longleftrightarrow\quad
 [\mathfrak b,T_{j,\mathrm M}]
 \in B(S_2(L_2(\R))),
\]
and, in this case
\begin{equation}\label{eq:plane-commutator-norm-equality}
 \|[b,T_j]\|_{B(L_2(\R_\theta^2))}
 =
 \|[\mathfrak b,T_{j,\mathrm M}]\|_{B(S_2(L_2(\R)))}.
\end{equation}
\end{lemma}

\begin{proof}
For $x\in W^{1,2}(\R_\theta^2)$, the Leibniz rule gives
$bx\in W^{1,2}(\R_\theta^2)$. Applying
\eqref{eq:plane-multiplier-unitary} with $m=m_j$, we obtain
\[
 \begin{aligned}
 \Phi([b,T_j]x)
 &=\Phi(bT_jx)-\Phi(T_j(bx))\\
 &=\mathfrak b\,T_{j,\mathrm M}\Phi(x)
   -T_{j,\mathrm M}\bigl(\mathfrak b\Phi(x)\bigr)\\
 &=[\mathfrak b,T_{j,\mathrm M}]\Phi(x).
 \end{aligned}
\]
Thus the two commutators are intertwined by $\Phi$. Moreover, by
\eqref{trace},
\[
 \|\Phi(x)\|_{S_2}=(2\pi)^{1/2}\|x\|_2.
\]
Hence $[b,T_j]$ extends to a bounded operator on $L_2(\R_\theta^2)$ if and only if
$[\mathfrak b,T_{j,\mathrm M}]$ extends to a bounded operator on
$S_2(L_2(\R))$.
Therefore,
\[
 \|[b,T_j]\|_{B(L_2(\R_\theta^2))}
 =
 \|[\mathfrak b,T_{j,\mathrm M}]\|_{B(S_2(L_2(\R)))}.
\]
\end{proof}

For notational convenience, we henceforth use the same notation $T_j$ for the operators on
$\R_\theta^2$ and their realizations on $B(L_2(\R))$, suppressing the subscript
$\mathrm M$. In particular, on $B(L_2(\R))$, we write
\[
 T_2=\mathcal W^{-1}M_{m_2}\mathcal W,
\]
where
\[
 m_2(s,t)
 =\frac{t^2}{\sqrt{s^2+t^2}},
 \qquad (s,t)\ne(0,0),\qquad m_2(0,0)=0.
\]
By Lemma \ref{lem:plane-commutator-equivalence}, it therefore remains to construct a symbol $\mathfrak b\in B(L_2(\R))$ such that
\[
 \|[P,\mathfrak b]\|_{B(L_2(\R))}<\infty,
 \qquad
 \|[Q,\mathfrak b]\|_{B(L_2(\R))}<\infty,
\]
while the commutator $[\mathfrak b,T_2]$ does not extend to a bounded operator on
$S_2(L_2(\R))$. 

\subsection{Construction of the counterexample on the Moyal plane}\label{s3.3}
We now construct the desired symbol $\mathfrak b$. The elementary estimate
\begin{equation}\label{eq:plane-absolute-approximation}
 |m_2(s,r+u)-|r||\leq |s|+|u|
\end{equation}
will be used repeatedly. Indeed,
\[
 0\leq |t|-m_2(s,t)\leq |s|,
\]
and hence \eqref{eq:plane-absolute-approximation} follows from
\[
 \big||r+u|-|r|\big|\leq |u|.
\]

Fix $a>0$ and a real-valued function
$\varphi\in C_c^\infty(\R)$ such that
\[
 \operatorname{supp}(\varphi)\subset(-a,a),
 \qquad
 \|\varphi\|_2=1.
\]
For $q\in\R$, write $\varphi_q(x)=\varphi(x-q)$, and define
\begin{equation}\label{eq:plane-packet-profile}
 h(s,u)
 :=e^{-\mathrm i su/2}
   \mathcal W(\varphi\otimes\varphi^*)(s,u)=(2\pi)^{-1/2}
 \int_\R e^{-\mathrm i s y}\varphi(y-u)\varphi(y)\,\dd y.
\end{equation}
Then $h(s,u)=0$ whenever $|u|\geq2a$, and $h$ decays rapidly in $s$,
uniformly for $u$ in its support. Moreover, by the unitarity of $\mathcal W$ and the regularity of $h$,
\begin{equation}\label{eq:plane-profile-moments}
 \int_{\R^2}|h(s,u)|^2\,\dd s\dd u=1,
 \qquad
 C_\varphi
 :=\int_{\R^2}(|s|+|u|)|h(s,u)|^2\,\dd s\dd u<\infty.
\end{equation}
Let $\mathsf D_n,\mathsf A_n$ be the matrices given by
Corollary~\ref{cor:scaled-obstruction}, and write
\[
 \mathsf D_n
 =\operatorname{diag}(\lambda_{n,1},\ldots,\lambda_{n,2N_n}).
\]
Thus
\begin{equation}\label{eq:plane-matrix-bounds}
 \|[\mathsf D_n,\mathsf A_n]\|\leq2^{-n},
 \qquad
 \|[|\mathsf D_n|,\mathsf A_n]\|\geq n.
\end{equation}
Choose $L_n>0$ sufficiently large so that
\begin{equation}\label{eq:plane-Ln}
 L_n\min_{k\ne l}|\lambda_{n,k}-\lambda_{n,l}|>4a,
 \qquad
 \frac{\|\mathsf A_n\|}{L_n}\leq2^{-n}.
\end{equation}
Next choose $z_n\in\R$ recursively so that all intervals
\[
 [q_{n,k}-a,q_{n,k}+a],
 \qquad
 q_{n,k}=z_n+L_n\lambda_{n,k},
\]
are pairwise disjoint, both within each generation and across different
generations. Consequently, the supports of the functions
$\varphi_{q_{n,k}}$ are pairwise disjoint, and
$\{\varphi_{q_{n,k}}\}_{n,k}$ is an orthonormal family in $L_2(\R)$.

Define the isometry
\[
 V_n:\C^{2N_n}\longrightarrow L_2(\R),
 \qquad
 V_ne_k=\varphi_{q_{n,k}},
\]
and set
\begin{equation}\label{eq:plane-blocks}
 G_n=L_n^{-1}\mathsf A_n,\qquad
 \mathfrak b_n=V_nG_nV_n^*,\qquad
 \mathfrak b=\sum_{n\geq1}\mathfrak b_n.
\end{equation}
The operators $\mathfrak b_n$ are self-adjoint and finite rank, and act on
mutually orthogonal subspaces of $L_2(\R)$. Moreover,
\[
 \|\mathfrak b_n\|=\|G_n\|
 \leq2^{-n}.
\]
Hence $\mathfrak b$ is the norm limit of finite-rank operators and is therefore
self-adjoint and compact.

Now we show that $\|\partial_j \mathfrak b\|<\infty,j=1,2$. Put
\[
 E_n^Qe_k=(x-q_{n,k})\varphi_{q_{n,k}},\qquad
 E_n^Pe_k=P\varphi_{q_{n,k}}.
\]
Disjointness of the packet supports makes the columns of each
of $E_n^Q$ and $E_n^P$ orthogonal. Hence
$\|E_n^Q\|=\|x\varphi\|_2$ and
$\|E_n^P\|=\|P\varphi\|_2$. Moreover, a direct calculation shows that
\[
 QV_n=V_n(z_n \mathbf{Id}_{2N_n}+L_n\mathsf D_n)+E_n^Q,
 \qquad PV_n=E_n^P.
\]
Hence we obtain
\begin{align}
 [Q,\mathfrak b_n]&=V_n[\mathsf D_n,\mathsf A_n]V_n^*
           +E_n^QG_nV_n^*-V_nG_n(E_n^Q)^*,\label{eq:plane-Q-block}\\
 [P,\mathfrak b_n]&=E_n^PG_nV_n^*-V_nG_n(E_n^P)^*.
 \label{eq:plane-P-block}
\end{align}
Consequently,
\begin{equation}\label{eq:plane-derivative-bounds}
 \|[Q,\mathfrak b_n]\|\leq (1+2\|x\varphi\|_2)2^{-n},\qquad
 \|[P,\mathfrak b_n]\|\leq2\|P\varphi\|_2\,2^{-n}.
\end{equation}
The series of both derivatives converge in norm. Closedness of $\partial_1,\partial_2$ on $K(L_2(\R))$ therefore shows that
$\mathfrak b=\mathfrak b^*\in K(L_2(\R))$ and $\partial_1\mathfrak b,\partial_2\mathfrak b\in K(L_2(\R)).$

We next show that $[\mathfrak b,T_2]$ is unbounded on $S_2(L_2(\R))$. Define isometries
\begin{equation}\label{eq:plane-Jn}
 J_n:\C^{2N_n}\longrightarrow S_2(L_2(\R)),\qquad
 J_ne_k=\varphi_{q_{n,k}}\otimes\varphi_{z_n}^*.
\end{equation}
The packet supports give the exact intertwining identity
\begin{equation}\label{eq:plane-exact-left-action}
 \mathfrak bJ_n=J_nG_n,
 \qquad J_n^*\mathfrak b=G_nJ_n^*.
\end{equation}
A change of variables in \eqref{eq:plane-weyl-plancherel} yields
\begin{equation}\label{eq:plane-packet-transform}
 \mathcal W(\varphi_q\otimes\varphi_z^*)(s,t)
 =e^{-\mathrm is(q-t/2)}h(s,t-q+z).
\end{equation}
For each fixed $n$, these transforms have pairwise disjoint $t$-supports.
Hence 
\[
\begin{aligned}
&\left\langle
T_2\bigl(\varphi_{q_{n,k}}\otimes\varphi_{z_n}^*\bigr),
\varphi_{q_{n,l}}\otimes\varphi_{z_n}^*
\right\rangle_{S_2}\\
&\qquad=
\int_{\R^2}
m_2(s,t)\,
e^{-\mathrm is(q_{n,k}-t/2)}
h(s,t-q_{n,k}+z_n)\,
\overline{
e^{-\mathrm is(q_{n,l}-t/2)}
h(s,t-q_{n,l}+z_n)
}
\,\dd s\dd t\\
&\qquad=
\delta_{kl}
\int_{\R^2}
m_2(s,u+L_n\lambda_{n,k})
|h(s,u)|^2\,\dd s\dd u.
\end{aligned}
\]
and therefore
$$J_n^*T_2J_n=\diag\left(\beta_{n,1},\ldots,\beta_{n,2N_n}\right)$$ with
\[
\beta_{n,k}=\int_{\R^2}m_2(s,u+L_n\lambda_{n,k})|h(s,u)|^2\,\dd s\dd u,\qquad k=1,\ldots,2N_n.
\]
From \eqref{eq:plane-absolute-approximation} and
\eqref{eq:plane-profile-moments}, we can find a $2N_n\times2N_n$ matrix $F_n$ such that
\begin{equation}\label{eq:plane-compression}
 J_n^*T_2J_n=L_n|\mathsf D_n|+F_n,
 \qquad F_n=F_n^*,\quad \|F_n\|\leq C_\varphi.
\end{equation}
By
\eqref{eq:plane-exact-left-action},
\begin{align}
 J_n^*[\mathfrak b,T_2]J_n
 &=[G_n,L_n|\mathsf D_n|+F_n]
   =[\mathsf A_n,|\mathsf D_n|]+[G_n,F_n].
 \label{eq:plane-commutator-compression}
\end{align}
Therefore,
\begin{equation}\label{eq:plane-blowup}
 \|J_n^*[\mathfrak b,T_2]J_n\|
 \geq n-2C_\varphi2^{-n}\longrightarrow\infty,
\end{equation}
which proves that $[\mathfrak b,T_2]$ is unbounded.

\subsection{Reduction to arbitrary deformation}
We now show that the same construction applies to every nonzero real
antisymmetric matrix $\theta$, including the degenerate case. Suppose
that
\(
 \operatorname{rank}\theta=2r>0,
\)
and put
\[
 J=
 \begin{pmatrix}
 0&1\\
 -1&0
 \end{pmatrix},
 \qquad
 \theta_0=J^{\oplus r}\oplus0_{d-2r}.
\]
There exists at least one invertible real $d\times d$ matrix
$T$ such that
\(
 \theta=T^{\mathsf T}\theta_0T.
\)
Write
\[
 T^{-1}(s_1,t_1,\ldots,s_r,t_r,\eta)
 =
 \sum_{a=1}^r(s_au_a+t_aw_a)+B\eta,
 \qquad
 \eta\in\R^{d-2r}.
\]

\begin{lemma}\label{T}
There exists an invertible real matrix $T$ satisfying
$\theta=T^{\mathsf T}\theta_0T$ such that the second column $w_1=((w_1)_1,\ldots,(w_1)_d)$ of $T^{-1}$ satisfies 
\[
 (w_1)_j\ne0,\qquad 1\leq j\leq d.
\]
\end{lemma}
\begin{proof}
  We identify $\theta$ with its associated bilinear form. Since $\theta\ne0$, $$\mu\left(\ker\theta\cup
 \bigcup_{j=1}^d\{x\in\R^d:x_j=0\}\right)=0.$$
    Hence we may choose
\[
 w_1\notin
 \ker\theta\cup
 \bigcup_{j=1}^d\{x\in\R^d:x_j=0\}.
\]
Since $w_1\notin\ker\theta$, there exists $u_1\in\R^d$ such that
\(
 u_1^{\mathsf T}\theta w_1=1.
\)
Let
\[
 F=\{v\in\R^d:u_1^{\mathsf T}\theta v
                    =w_1^{\mathsf T}\theta v=0\}.
\]
Then $\R^d=\operatorname{span}\{u_1,w_1\}\oplus F$, and  $\theta\mid F$ has rank $2r-2$ and radical $\ker\theta$. We may choose $r-1$ symplectic pairs $(u_2,w_2),\dots,(u_r,w_r)$ in $F$ such that
$$
 u_j^{\mathsf T}\theta w_k=\delta_{jk},
 \qquad
 u_j^{\mathsf T}\theta u_k=0,
 \qquad
 w_j^{\mathsf T}\theta w_k=0,\qquad  2\le j,k\le r.
$$
and complete them by a basis of $\ker\theta$. Together with
$u_1,w_1$, these vectors form the columns of an invertible matrix
$S$ satisfying $S^{\mathsf T}\theta S=\theta_0$. Taking
$T=S^{-1}$ proves the assertion.
\end{proof}

Let $T$ be the matrix in Lemma \ref{T}.
By \cite[Proposition 2.3]{GJM}, the change of variables
\[
 \Psi_T:\R_\theta^d\longrightarrow\R_{\theta_0}^d,
 \qquad
 \Psi_T(\lambda_\theta(\xi))
 =\lambda_{\theta_0}(T\xi),
\]
extends to a normal $*$-isomorphism. Combining this with the identification
$\R_J^2\simeq B(L_2(\R))$, we obtain a normal $*$-isomorphism
\begin{equation}\label{eq:reduction-tensor-isomorphism}
\begin{aligned}
     &\Phi_T:\R_\theta^d
 \longrightarrow
 \overline{\bigotimes}_{a=1}^r B(L_2(\R))
 \,\overline\otimes\,
 L_\infty(\R^{d-2r}),\\ 
 &\Phi_T(\lambda_\theta(\xi))(y)=
  e^{\mathrm i\eta\cdot y}W(s_1,t_1)\otimes\cdots\otimes W(s_r,t_r)
\end{aligned}
\end{equation}
with $T\xi=(s_1,t_1,\ldots,s_r,t_r,\eta).$
Under this isomorphism, the canonical trace is given by
\[
 \tau_\theta(x)
 =
 \frac{|\det T|}{(2\pi)^{d-r}}
 \left(
  \operatorname{Tr}^{\otimes r}
  \otimes
  \int_{\R^{d-2r}}
 \right)\Phi_T(x).
\]

Let $\delta_1,\ldots,\delta_d$ denote the canonical derivations on the
right-hand side of \eqref{eq:reduction-tensor-isomorphism}. Explicitly, for $1\leq a\leq r$, they are given on the $a$th Moyal factor by
\[
 \delta_{2a-1}
 =\mathrm i[P_a,\,\cdot\,],
 \qquad
 \delta_{2a}
 =\mathrm i[Q_a,\,\cdot\,],
\]
while on the commutative factor,
\[
 \delta_{2r+\ell}
 =\partial_{y_\ell},
 \qquad 1\leq\ell\leq d-2r.
\]
Then the
original derivations $\partial_k$ on $\R_\theta^d$ become
\begin{equation}\label{eq:reduction-derivations}
 \Phi_T(\partial_kx)
 =
 \sum_{\ell=1}^d
 (T^{-1})_{k\ell}\,
 \delta_\ell\Phi_T(x),
 \qquad 1\leq k\leq d.
\end{equation}

To describe the transformed Calder\'on commutators, let
$\mathfrak b$ be the counterexample constructed in
Section~\ref{s3.3} and $1$ be the unit of $L_\infty(\R^{d-2r})$, and define
\[
 b
 =
 \Phi_T^{-1}\bigl(\mathfrak b^{\otimes r}\otimes1\bigr).
\]
Since $\mathfrak b$ is self-adjoint and has bounded first-order
derivations, it follows from \eqref{eq:reduction-derivations} that
$b$ is self-adjoint and belongs to $W^{1,\infty}(\R_\theta^d)$. More
precisely,
\[
 \sup_{1\leq k\leq d}\|\partial_k b\|
 \lesssim_T
 \|\mathfrak b\|^{r-1}
 \bigl(
   \|[P,\mathfrak b]\|
   +\|[Q,\mathfrak b]\|
 \bigr).
\]
On the other hand, under the tensor-product realization
\eqref{eq:reduction-tensor-isomorphism}, the corresponding Fourier
transform is given by
\(
 \bigotimes_{i=1}^r \mathcal W
 \otimes \mathcal F_{\R^{d-2r}},
\)
and $\Phi_T T_j\Phi_T^{-1}$ is the Fourier multiplier with symbol
\begin{equation}\label{eq:reduction-multiplier}
 (s_1,t_1,\ldots,s_r,t_r,\eta)
 \longmapsto
 m_j\left(
 \sum_{a=1}^r(s_au_a+t_aw_a)+B\eta
 \right).
\end{equation}
Since the function
\[
 m_j(\xi)=\frac{\xi_j^2}{|\xi|},
 \qquad \xi\ne0,\qquad m_j(0)=0
\]
has Lipschitz constant at most $3$ and satisfies
\[
 m_j(\rho w_1)
 =
 \frac{(w_1)_j^2}{|w_1|}\,|\rho|,
\]
we obtain, for all $\rho,s,v\in\R$,
\begin{align}
&\left|
 m_j\left(
 su_1+(\rho+v)w_1
 +\sum_{a=2}^r(s_au_a+t_aw_a)
 +B\eta
 \right)
 -
 \frac{(w_1)_j^2}{|w_1|}|\rho|
\right|
\notag\\
&\qquad\leq
3\left(
 |u_1||s|
 +|w_1||v|
 +\sum_{a=2}^r
   \bigl(|u_a||s_a|+|w_a||t_a|\bigr)
 +\|B\||\eta|
\right).
\label{eq:reduction-symbol-estimate}
\end{align}

Fix $\psi\in\mathcal S(\R^{d-2r})$ with
\(
 \|\psi\|_{L_2(\R^{d-2r})}=1.
\)
When $d=2r$, the factor $\psi$ is absent. Consider the isometry
\[
 E_\psi:\C\longrightarrow L_2(\R^{d-2r}),
 \qquad
 c\longmapsto c\psi,
\]
define
\[
 \widetilde J_n
 =
 J_n\otimes J_1^{\otimes(r-1)}\otimes E_\psi.
\]
Since
\[
 \Phi_T(b)=\mathfrak b^{\otimes r}\otimes1,
\]
the exact block action \eqref{eq:plane-exact-left-action} gives
\begin{align}\label{BJ}
     \Phi_T(b)\widetilde J_n
 =
 \widetilde J_n
 \bigl(G_n\otimes G_1^{\otimes(r-1)}\otimes 1\bigr),\qquad
 \widetilde J_n^*\Phi_T(b)
 =
 \bigl(G_n\otimes G_1^{\otimes(r-1)}\otimes 1\bigr)
 \widetilde J_n^*.
\end{align}
By
\eqref{eq:reduction-symbol-estimate}, arguing as in \eqref{eq:plane-compression}, we obtain operators $F_{n,j}$ and a
constant $C$, independent of $n$ and $j$, such that
\begin{equation}\label{eq:reduction-compression}
 \widetilde J_n^*
 \Phi_TT_j\Phi_T^{-1}
 \widetilde J_n
 =
 \frac{(w_1)_j^2L_n}{|w_1|}
 \bigl(
 |\mathsf D_n|\otimes \mathbf{Id}_{2N_1}^{\otimes(r-1)}\otimes 1
 \bigr)
 +F_{n,j},
 \qquad
 \|F_{n,j}\|\leq C,
\end{equation}
Set
\[
 H_n:=G_n\otimes G_1^{\otimes(r-1)}\otimes 1.
\]
Using \eqref{BJ}, \eqref{eq:plane-matrix-bounds} and \eqref{eq:plane-Ln}, we obtain
\begin{align*}
 \left\|
 \widetilde J_n^*
 [\Phi_T(b),\Phi_TT_j\Phi_T^{-1}]
 \widetilde J_n
 \right\|
 &=\left\|
 \left[H_n,
 \frac{(w_1)_j^2L_n}{|w_1|}
 \bigl(|\mathsf D_n|\otimes \mathbf{Id}_{2N_1}^{\otimes(r-1)}\otimes 1\bigr)+F_{n,j}
 \right]\right\|\\
 &\geq
 \|G_1\|^{r-1}
 \left(
 \frac{(w_1)_j^2}{|w_1|}
 \|[\mathsf A_n,|\mathsf D_n|]\|
 -2C\|G_n\|
 \right)\\
 &\geq
 \|G_1\|^{r-1}
 \left(
 \frac{(w_1)_j^2}{|w_1|}n
 -2C2^{-n}
 \right)
 \longrightarrow\infty,
\end{align*}
since $G_1\ne0$ by \eqref{eq:plane-matrix-bounds} and
$(w_1)_j\ne0$ by Lemma~\ref{T}. Since $\Phi_T$ induces an
isomorphism of the corresponding $L_2$-spaces up to a fixed trace
normalization, it follows that $[b,T_j]$ is unbounded for every
$j=1,\ldots,d$.
This proves Theorem~\ref{thm:euclidean}.

\section{Proof of Theorem~\ref{thm:torus}}\label{sec:torus}
In this section, we prove Theorem~\ref{thm:torus}.
We first construct a counterexample on the two-dimensional
irrational rotation algebra, using disjoint orbit towers to realize the finite-dimensional obstruction from Section~\ref{s2}. The construction yields a self-adjoint symbol with bounded first derivatives but an unbounded Calder\'on's first commutator.
We then embed an irrational rotation algebra into any
higher-dimensional quantum torus with an irrational deformation
parameter, choosing an appropriate integer sublattice so that
the same symbol produces unbounded commutators in every coordinate.

\subsection{Quantum tori and their differential structure}

We recall the definitions needed for the quantum torus arguments, which are similar to those in the Euclidean case; see
\cite{ChenXuYin2013,XiongXuYin2018}. Let $\theta$ be a real
antisymmetric $d\times d$ matrix. On $\ell_2(\Z^d)$, with standard
basis $\{\varepsilon_n:n\in\Z^d\}$, define
\[
 U_j\varepsilon_n
 =\exp\left(2\pi\mathrm i\sum_{k<j}\theta_{jk}n_k\right)
       \varepsilon_{n+e_j},
 \qquad 1\leq j\leq d,
\]
where $e_j$ denotes the $j$th coordinate vector in $\Z^d$. These
unitaries satisfy
\[
 U_kU_l=e^{2\pi\mathrm i\theta_{kl}}U_lU_k.
\]
The quantum torus is the finite von Neumann algebra
\[
 \T_\theta^d=\{U_1,\ldots,U_d\}''
 \subset B(\ell_2(\Z^d)).
\]
Its canonical normal faithful trace $\tau$ is determined by
\[
 \tau(U^n)=\delta_{n,0},
 \qquad
 U^n=U_1^{n_1}\cdots U_d^{n_d},
 \qquad n\in\Z^d.
\]
This trace gives rise to the associated noncommutative $L_p$-spaces
$L_p(\T_\theta^d)$, $1\leq p<\infty$, while
$L_\infty(\T_\theta^d):=\T_\theta^d$.
The family $\{U^n:n\in\Z^d\}$ forms an orthonormal basis of
$L_2(\T_\theta^d)$. Writing
\[
 \widehat x(n)=\tau\bigl((U^n)^*x\bigr),
\]
one has the Plancherel identity
\begin{equation}\label{eq:torus-fourier-plancherel}
 \|x\|_2^2=\sum_{n\in\Z^d}|\widehat x(n)|^2.
\end{equation}
The smooth subalgebra $C^\infty(\T_\theta^d)$ consists of $x\in\T_\theta^d$
whose Fourier coefficients $\widehat x(n)$ are rapidly decreasing, namely,
\[
 \sup_{n\in\Z^d}(1+|n|_1)^N|\widehat x(n)|<\infty
 \qquad \forall N\geq0.
\]

The canonical torus action is the trace-preserving normal action
defined by
\[
 \alpha_s(U^n)=e^{\mathrm i n\cdot s}U^n,
 \qquad s\in\T^d.
\]
It is strongly continuous on $L_2(\T_\theta^d)$ and weak-$*$ continuous
on $\T_\theta^d$. On $C^\infty(\T_\theta^d)$, the canonical partial
derivatives are given by
\[
 \partial_kx
 =\sum_{n\in\Z^d}\mathrm i n_k\widehat x(n)U^n,
 \qquad 1\leq k\leq d.
\]
For a multi-index $\beta\in\mathbb N^d$, we write
\[
 \partial^\beta
 =\partial_1^{\beta_1}\cdots\partial_d^{\beta_d}.
\]

For an integer $r\geq0$ and $1\leq p\leq\infty$, the Sobolev space
$W^{r,p}(\T_\theta^d)$ and its norm is defined by
\begin{equation}\label{eq:torus-first-sobolev}
 \begin{aligned}
 W^{r,p}(\T_\theta^d)
   &=\{x\in L_p(\T_\theta^d):
        \partial^\beta x\in L_p(\T_\theta^d)
        \text{ for }|\beta|_1\leq r\},\\
 \|x\|_{W^{r,p}}
   &=\sum_{|\beta|_1\leq r}\|\partial^\beta x\|_p.
 \end{aligned}
\end{equation}
In particular, $W^{1,\infty}(\T_\theta^d)$ agrees with the Lipschitz
class defined in \eqref{eq:W1infty-definition2}. Moreover, regarding
each $\partial_k$ as its closed realization on $L_2(\T_\theta^d)$,
\[
 W^{1,2}(\T_\theta^d)
 =\bigcap_{k=1}^d\operatorname{Dom}(\partial_k).
\]

For a symbol $m:\Z^d\to\C$, the Fourier multiplier $m(D)$ is defined by
\[
 m(D)\left(\sum_{n\in\Z^d}a_nU^n\right)
 =\sum_{n\in\Z^d}m(n)a_nU^n.
\]
As in the Euclidean setting, put
\[
 L=-\sum_{k=1}^d\partial_k^2,\qquad
 R_j=\partial_jL^{-1/2},\qquad
 T_j=-\partial_jR_j.
\]
Then $T_j$ is the Fourier multiplier with symbol
$$m_j(n)=\frac{n_j^2}{|n|},n\ne 0,\qquad m_j(0)=0.$$
For $b\in W^{1,\infty}(\T_\theta^d)$, the same
Leibniz-rule argument as in the Euclidean case shows that
Calder\'on's first commutator $[b,T_j]$ is well defined on $W^{1,2}(\T_\theta^d)$.

\subsection{The two-dimensional construction}\label{s4.2}
We first consider $d=2$, where
\[
 \theta=
 \begin{pmatrix}
  0&-\vartheta\\
  \vartheta&0
 \end{pmatrix},
 \qquad
 \text{for some}\ \vartheta\in\R\setminus\mathbb Q.
\]
Then the two-dimensional quantum torus can be identified with the
irrational rotation von Neumann algebra
\[
 \mathsf A_\vartheta=\{\mathsf U,\mathsf V\}'',
 \qquad
 \mathsf V\mathsf U
 =e^{2\pi\mathrm i\vartheta}\mathsf U\mathsf V,
 \qquad
 \tau(\mathsf U^m\mathsf V^p)
 =\delta_{m,0}\delta_{p,0}.
\]
The derivatives are given by
\[
 \delta_1(\mathsf U^m\mathsf V^p)
 =\mathrm i m\mathsf U^m\mathsf V^p,
 \qquad
 \delta_2(\mathsf U^m\mathsf V^p)
 =\mathrm i p\mathsf U^m\mathsf V^p.
\]
The abelian von Neumann algebra $\{\mathsf U\}''$ is identified with
$L_\infty(\T)$ through the spectral calculus of $\mathsf U$, with the
canonical trace corresponding to the integration $\int_{\T}$. For
$f\in L_\infty(\T)$, we write $f(\mathsf U)$ for the corresponding functional
calculus and define 
\[
 \sigma_pf(x)=f(x+2\pi p\vartheta).
\]
Then $ \operatorname{supp}(\sigma_pf)
 =\operatorname{supp}(f)-2\pi p\vartheta$ and
\[
 \mathsf V^pf(\mathsf U)\mathsf V^{-p}=\sigma_pf(\mathsf U).
\]

The following elementary packing observation allows us to choose the
spatial profiles before selecting the large Fourier frequencies.

\begin{lemma}\label{lem:orbit-packing}
Assume that $\vartheta\in\R\setminus\mathbb Q.$ Let $q_1,\ldots,q_M$ be distinct integers, and let $\Omega$ be a nonempty
proper open arc of $\T$. There exist $\eta>0$ and an unbounded set
$\mathcal L\subset\mathbb N$ such that, for every $\ell\in\mathcal L$,
there is a compact arc $I_\ell$ of length $\eta$ for which the arcs
\[
 I_\ell-2\pi\ell q_r\vartheta,\qquad 1\leq r\leq M,
\]
are pairwise disjoint and contained in $\Omega$.
\end{lemma}

\begin{proof}
Let $\pi:\R\to\T$ be the quotient map. Choose a lift of $\Omega$ of the form
\[
 \widetilde\Omega=(c-\omega/2,c+\omega/2)\subset\R
\]
such that
\(
 \pi(\widetilde\Omega)=\Omega.
\) Choose $t_0>0$ sufficiently small that
\[
 (\max_r q_r-\min_r q_r)t_0<\omega/4.
\]
The numbers $q_rt_0$ are distinct. Hence, there exists an open interval $J$
containing $t_0$ and contained in $(0,2\pi)$, and $\eta>0$ such that
\[
 \min_{r\neq s}|q_r-q_s|t>\eta,
 \qquad
 (\max_r q_r-\min_r q_r)t+\eta/2<\omega/2
 \quad(t\in J).
\]
If $K=[c-\eta/2,c+\eta/2]$, the intervals
$K+(q_1-q_r)t,\,1\le r\le M$ are therefore disjoint and lie in $\widetilde\Omega$.
Moreover,
\[
 \pi(K+(q_1-q_r)t)=I_t-q_rt,
\]
where $I_t:=\pi(K+q_1t)$.
Since $\{2\pi\ell\vartheta\bmod 2\pi:\ell\in\mathbb N\}$ is dense in $(0,2\pi)$, there exist
arbitrarily large positive integers $\ell$ such that
\(
 2\pi\ell\vartheta\pmod {2\pi}
\)
belongs to $J$. Replacing $t_0$ with $2\pi\ell\vartheta\bmod 2\pi$ and repeating the
preceding construction yields an interval $I_\ell$ with the desired
properties.
\end{proof}

We next record the matrix realization supplied by disjoint translates.
Suppose that $p_1,\ldots,p_M$ are distinct integers and that the arcs
$I-2\pi p_r\vartheta$ are pairwise disjoint. Choose real smooth functions $h,g$
on $\T$ satisfying
\begin{equation}\label{eq:torus-profiles}
 0\leq h\leq1,\qquad
 \operatorname{supp}h\subset I,\qquad
 \operatorname{supp}g\subset\{h=1\},\qquad \|g\|_2=1.
\end{equation}
Put
\[
 w_r=\mathsf V^{p_r}h(\mathsf U),\qquad
 \mathcal R=[w_1\ \cdots\ w_M],\qquad
 \Pi=\operatorname{diag}(p_1,\ldots,p_M).
\]
Disjointness gives
\begin{equation}\label{eq:torus-row}
 w_r^*w_s=\delta_{rs}h(\mathsf U)^2,
 \qquad \mathcal R^*\mathcal R=h(\mathsf U)^2I_M,
 \qquad \|\mathcal R\|\leq1.
\end{equation}
Indeed, the first identity of \eqref{eq:torus-row} follows from the fact that
$$w_r^*w_s=\mathsf V^{-p_r}\sigma_{p_r}(h)(\mathsf U)\sigma_{p_s}(h)(\mathsf U)\mathsf V^{p_s}.$$
For an $M\times M$ Hermitian matrix $C=C^*$, the element
$\mathcal RC\mathcal R^*$ is self-adjoint and satisfies
\begin{align}
 \|\mathcal RC\mathcal R^*\|&\leq\|C\|,
 \label{eq:torus-block-norm}\\
 \|\delta_1(\mathcal RC\mathcal R^*)\|
 &\leq 2\|h'\|_\infty\|C\|,
 \label{eq:torus-block-d1}\\
 \delta_2(\mathcal RC\mathcal R^*)
 &=\mathrm i\mathcal R[\Pi,C]\mathcal R^*.
 \label{eq:torus-block-d2}
\end{align}
Here the last identity follows from
$\delta_2\mathcal R=\mathrm i\mathcal R\Pi$.

Now choose pairwise disjoint nonempty open arcs $\Omega_n\subset\T$,
and set
\(
 a_n=2^{-n-1}/\pi^2.
\)
Using Proposition \ref{prop:finite-obstruction}, choose $N_n$ so large that
\begin{equation}\label{eq:torus-growth}
 a_n\kappa_{N_n}\geq4n+2.
\end{equation}
Fix $n$ and let $q_{n,1},\ldots,q_{n,2N_n}$ be the diagonal entries of $D_{N_n}$;
thus the ordered list is
\[
 2,2^2,\ldots,2^{N_n},-2,-2^2,\ldots,-2^{N_n}.
\]
Applying Lemma~\ref{lem:orbit-packing} to these distinct integers and $\Omega_n$ gives a length $\eta_n>0$ and an unbounded set $\mathcal L_n$ such that for every $\ell\in \mathcal L_n,$ there exists a compact arc $I_\ell$ of length $\eta_n$ such that
\[
 I_\ell-2\pi\ell q_{n,r}\vartheta,\qquad 1\leq r\leq 2N_n,
\]
are pairwise disjoint and contained in $\Omega_n$.
Fix $\ell_{0,n}\in\mathcal L_n$. Choose profiles $h_n^0,g_n^0$
as in~\eqref{eq:torus-profiles} with $I=I_{\ell_{0,n}}$, and set
\[
 H_n:=\|(h_n^0)'\|_\infty,
 \qquad
 \mu_n:=\sum_{m\in\mathbb Z}|m|\,|\widehat{g_n^0}(m)|^2<\infty.
\]
Choose $\ell_n\in\mathcal L_n$ sufficiently large that
\begin{equation}\label{eq:torus-scale}
 \frac{\pi a_n}{6\ell_n}\leq2^{-n},
 \qquad
 \frac{\pi a_nH_n}{3\ell_n}\leq2^{-n},
 \qquad
 \frac{\pi a_n\mu_n}{3\ell_n}\leq1.
\end{equation}
Translate the arc $I_{\ell_{0,n}}$ to $I_{\ell_n}$, and let
$h_n,g_n$ denote the corresponding translates of $h_n^0,g_n^0$.
Translation preserves $H_n$, $\mu_n$, and all the properties in
\eqref{eq:torus-profiles}. This order of choices is important: the
profile estimates are fixed before $\ell_n$ is chosen sufficiently
large.

Define
\begin{gather*}
 p_{n,r}=\ell_nq_{n,r},\qquad
 \Pi_n=\ell_nD_{N_n},\qquad
 \Gamma_n=\frac{a_n}{\ell_n}E_{N_n},\\
 \mathcal R_n=[\mathsf V^{p_{n,1}}h_n(\mathsf U)\ \cdots\
 \mathsf V^{p_{n,2N_n}}h_n(\mathsf U)],
\end{gather*}
and put $\mathfrak b_n=\mathcal R_n\Gamma_n\mathcal R_n^*$. By construction,
the supports of all $\sigma_{p_{n,r}}h_n$ lie in $\Omega_n$ and are
pairwise disjoint. The same is true of $\sigma_{p_{n,r}}g_n$.
By Proposition \ref{prop:finite-obstruction} and
\eqref{eq:torus-block-norm}--\eqref{eq:torus-scale},
\begin{align*}
 \|\mathfrak b_n\|&\leq\frac{\pi a_n}{6\ell_n}\leq2^{-n},\\
 \|\delta_1\mathfrak b_n\|&\leq\frac{\pi a_nH_n}{3\ell_n}
 \leq2^{-n},\\
 \|\delta_2\mathfrak b_n\|&\leq a_n\|[D_{N_n},E_{N_n}]\|
 \leq\pi a_n\leq2^{-n}.
\end{align*}
Consequently all three series $\sum_n\mathfrak b_n$, $\sum_n\delta_1\mathfrak b_n$,
and $\sum_n\delta_2\mathfrak b_n$ converge in the operator norm of
$\mathsf A_\vartheta$, which gives
\begin{equation}\label{eq:torus-symbol}
 \mathfrak b=\sum_{n\geq1}\mathfrak b_n=\mathfrak b^*\in \mathsf A_\vartheta,
 \qquad
 \delta_k\mathfrak b=\sum_{n\geq1}\delta_k\mathfrak b_n\in \mathsf A_\vartheta
 \quad(k=1,2).
\end{equation}

It remains to show that the commutator is unbounded. Define isometries
\[
 J_n:\mathbb C^{2N_n}\longrightarrow L_2(\mathsf A_\vartheta),
 \qquad J_ne_r=\mathsf V^{p_{n,r}}g_n(\mathsf U).
\]
The vectors $\mathsf V^{p_{n,r}}g_n(\mathsf U)$ are
orthonormal due to the choice of $h_n$ and $g_n$. Moreover, we have the exact identities
\[
 (\mathsf V^{p_{n,r}}h_n(\mathsf U))^*J_ne_s=\delta_{rs}g_n(\mathsf U),
 \qquad \mathcal R_n^*J_ne_s=e_sg_n(\mathsf U)
\]
since $h_ng_n=g_n$, which yields $\mathfrak b_nJ_n=J_n\Gamma_n$. If $m\neq n$, each entry of
$\mathcal R_m^*J_n$ contains a product of a translate of $h_m$ and a
translate of $g_n$. These supports lie in the disjoint arcs
$\Omega_m$ and $\Omega_n$, respectively, so $\mathfrak b_mJ_n=0$. Taking the norm limit in
\eqref{eq:torus-symbol} gives
\begin{equation}\label{eq:torus-exact-action}
 \mathfrak bJ_n=J_n\Gamma_n,
 \qquad J_n^*\mathfrak b=\Gamma_nJ_n^*.
\end{equation}

Now by the identity $$g_n(\mathsf U)=\sum_{m\in\Z}\widehat{g_n}(m)\mathsf U^m,$$ we have
\[
 J_n^*T_2J_n
 =\operatorname{diag}\left(
 \sum_{m\in\mathbb Z}|\widehat g_n(m)|^2
 \frac{p_{n,r}^2}{\sqrt{m^2+p_{n,r}^2}}
 \right)_{r=1}^{2N_n}.
\]
For real $m,p$, by \eqref{eq:plane-absolute-approximation},
\[
 0\leq |p_{n,r}|-\frac{p_{n,r}^2}{\sqrt{m^2+p_{n,r}^2}}\leq |m|.
\]
Parseval's
identity therefore yields a real diagonal matrix $F_n$ such that
\begin{equation}\label{eq:torus-compression}
 J_n^*T_2J_n=|\Pi_n|+F_n,
 \qquad \|F_n\|\leq\mu_n,
\end{equation}
which combined with~\eqref{eq:torus-exact-action} implies that
\[
 J_n^*[\mathfrak b,T_2]J_n=[\Gamma_n,|\Pi_n|+F_n].
\]
Using~\eqref{eq:torus-growth} and~\eqref{eq:torus-scale}, we obtain
\begin{align*}
 \|J_n^*[\mathfrak b,T_2]J_n\|
 &\geq a_n\|[E_{N_n},|D_{N_n}|]\|
       -2\mu_n\|\Gamma_n\|\\
 &\geq\left(a_n\kappa_{N_n}
          -\frac{\pi a_n\mu_n}{3\ell_n}\right)
 \geq4n+1.
\end{align*}
Hence $[\mathfrak b,T_2]$ has no bounded
extension to $L_2(\mathsf A_\vartheta)$.
\subsection{Passage to arbitrary dimension and every coordinate}

Suppose that $\theta$ has an irrational entry. As in the Euclidean
argument, we choose a two-dimensional frequency plane so that the same
construction produces an obstruction in every coordinate.

Choose $p\ne q$ with $\theta_{qp}\notin\mathbb Q$, and put
$v^0=(1,\ldots,1)^{\mathsf T}$. At least one of the two vectors
$v=v^0+Ne_q$, $N\in\{0,1\}$, satisfies
\begin{equation}\label{eq:torus-lattice-choice}
 \vartheta:=v^{\mathsf T}\theta e_p\notin\mathbb Q.
\end{equation}
Fix such $v=(v_1,\dots,v_d)$. Then every $v_j$ is nonzero, and $e_p,v$ are linearly
independent because $(e_p)_q=0$ and $v_q\ne0$.

Now consider the unitaries $\mathsf U:=U_p$ and $\mathsf V:=U^v$. They satisfy
\[
 U^vU_p=e^{2\pi\mathrm i\vartheta}U_pU^v.
\]
Since the map
$(m,k)\mapsto me_p+kv$ is injective,
\[
 \tau\bigl(U_p^m(U^v)^k\bigr)=\delta_{m,0}\delta_{k,0}.
\]
Consequently, the map determined by
\begin{equation}\label{eq:torus-embedding}
 \Phi_\theta^\T:\mathsf A_\vartheta\longrightarrow\T_\theta^d,
 \qquad
 \Phi_\theta^\T(\mathsf U^m\mathsf V^k)=U_p^m(U^v)^k
\end{equation}
is a trace-preserving normal
$*$-homomorphism. In particular, $\Phi^\T_\theta$ induces an isometry of the corresponding
$L_2$ spaces.
Let $\delta_1,\delta_2$ denote the two canonical derivations on
$\mathsf A_\vartheta$. The ambient derivatives $\partial_j,1\le j\le d$ satisfy
\begin{equation}\label{eq:torus-embedded-derivatives}
 \partial_j\Phi_\theta^\T(x)
 =\Phi^\T_\theta\bigl(\delta_{pj}\delta_1x+v_j\delta_2x\bigr),
 \qquad 1\leq j\leq d.
\end{equation}
Likewise, the ambient multipliers $T_j$ satisfy
\begin{equation}\label{eq:torus-embedded-multiplier}
 T_j\Phi_\theta^\T(\mathsf U^m\mathsf V^k)
 =m_j\bigl(me_p+kv\bigr)\Phi_\theta^\T(\mathsf U^m\mathsf V^k).
\end{equation}

We now define
\[
 b_n:=\Phi_\theta^\T(\mathfrak b_n),
 \qquad
 b:=\Phi_\theta^\T(\mathfrak b),
\]
where $\mathfrak b_n$ is the element constructed in
Section~\ref{s4.2}. We use the same notation $p_{n,r}$, $g_n$, $J_n$, $\Pi_n$, and
$\Gamma_n$ as in Section~\ref{s4.2}. By
\eqref{eq:torus-embedded-derivatives} and the estimates established
there,
\[
 \|b_n\|\leq2^{-n},
 \qquad
 \|\partial_j b_n\|
 \leq
 \bigl(|\delta_{p,j}|+|v_j|\bigr)2^{-n},
 \qquad 1\leq j\leq d.
\]
Thus $b=\sum_{n\geq1}b_n$ and all its first derivatives converge in
operator norm, so that
\[
 b=b^*\in W^{1,\infty}(\T_\theta^d).
\]
Since $\Phi_\theta^\T$ preserves multiplication, 
\eqref{eq:torus-exact-action} gives
\[
 b\Phi_\theta^\T J_n=\Phi_\theta^\T J_n\Gamma_n,\qquad (\Phi_\theta^\T J_n)^*b=\Gamma_n(\Phi_\theta^\T J_n)^*.
\]
On the other hand, as in the Euclidean argument,
$m_j$ has Lipschitz constant at most $3$, and
$m_j(kv)=|k|v_j^2/|v|$. Hence
\begin{equation}\label{eq:torus-restricted-symbol}
 \left|m_j\bigl(me_p+kv\bigr)-|k|v_j^2/|v|\right|
 \leq3|m|,
 \qquad m,k\in\Z.
\end{equation}
The integer Fourier support of $\Phi_\theta^\T J_ne_r=\Phi_\theta^\T(\mathsf V^{p_{n,r}}g_n(\mathsf U))$ is contained in
$\{me_p+p_{n,r}v:m\in\Z\}$, and these sets are disjoint for distinct
$r$. Therefore \eqref{eq:torus-embedded-multiplier} gives
\[
 J_n^*(\Phi_{\theta}^\T)^* T_j\Phi_{\theta}^\T J_n
 =\operatorname{diag}\left(
   \sum_{m\in\Z}|\widehat g_n(m)|^2
     m_j\bigl(me_p+p_{n,r}v\bigr)
 \right)_{r=1}^{2N_n}.
\]
Using \eqref{eq:torus-restricted-symbol}, Parseval's identity, and
the definition of $\mu_n$, there exists a real diagonal matrix $F_{n,j}$ such that
\begin{equation}\label{eq:torus-highdim-compression}
 J_n^*(\Phi_{\theta}^\T)^* T_j\Phi_{\theta}^\T J_n=\frac{v_j^2}{|v|}|\Pi_n|+F_{n,j},
 \qquad
 \|F_{n,j}\|\leq3\mu_n.
\end{equation}
The exact block action now yields
\[
 J_n^*(\Phi_{\theta}^\T)^*[b,T_j]\Phi_{\theta}^\T J_n
 =[\Gamma_n,\frac{v_j^2}{|v|} |\Pi_n|+F_{n,j}].
\]
Since $\Gamma_n=(a_n/\ell_n)E_{N_n}$ and
$\Pi_n=\ell_nD_{N_n}$, Proposition \ref{prop:finite-obstruction} and \eqref{eq:torus-growth}--\eqref{eq:torus-scale} imply
\begin{align*}
 \| J_n^*(\Phi_{\theta}^\T)^*[b,T_j]\Phi_{\theta}^\T J_n\|
 &\geq a_n\kappa_{N_n}\frac{v_j^2}{|v|}
       -6\mu_n\|\Gamma_n\|\\
 &\geq (4n+2)\frac{v_j^2}{|v|}-3
 \longrightarrow\infty.
\end{align*}
Thus this self-adjoint $b$
makes $[b,T_j]$ unbounded on
$L_2(\T_\theta^d)$ for every $j=1,\ldots,d$.
This proves Theorem~\ref{thm:torus}.

\section{Proof of Theorem~\ref{thm:torus-rational}}\label{subsec:rational-tori}

In this section, we prove Theorem~\ref{thm:torus-rational}. We first prove a matrix-valued $T1$ estimate with logarithmic
dependence on the matrix dimension. The key ingredient in the proof is the operator-valued dyadic representation established in \cite[Theorem~2]{HanninenHytonen2016}.
We then use the transference argument to obtain the upper bound on a rational quantum torus. To prove optimality, we construct a smooth symbol using the
equivariance structure of the quantum torus that attains the logarithmic order of the upper bound.

\subsection{\texorpdfstring{Operator-valued $T1$ operators and dyadic representation}{Operator-valued T1 operators and dyadic representation}}

We first recall the definitions of operator-valued BMO spaces. For a separable Hilbert space $\mathcal H$, let $F:\R^d\to B(\mathcal H)$. The column and row $\mathrm{BMO}(\R^d;B(\mathcal H))$ seminorms are
defined by
\begin{align}
 \|F\|_{\mathrm{BMO}^c}
 &=
 \sup_Q
 \left\|
   \frac1{|Q|}
   \int_Q
   (F(x)-F_Q)^*(F(x)-F_Q)\,\dd x
 \right\|^{1/2},
 \label{eq:operator-BMO-column}\\
 \|F\|_{\mathrm{BMO}^r}
 &=
 \|F^*\|_{\mathrm{BMO}^c},
 \qquad
 \|F\|_{\mathrm{BMO}^{rc}}
 =
 \max\{
   \|F\|_{\mathrm{BMO}^c},
   \|F\|_{\mathrm{BMO}^r}
 \}.
 \label{eq:operator-BMO-row}
\end{align}
Here the supremum is taken over all cubes $Q\subset\R^d$, and
\[
 F_Q=\frac1{|Q|}\int_QF(x)\,\dd x.
\]
These are the BMO spaces arising in noncommutative harmonic analysis;
see Mei~\cite{Mei2007}. 

We then review the operator-valued $T1$ theory; see, for instance,
\cite{Isralowitz2018,HongLiuMei2020,Hytonen2021}.
Let $T$ be a $B(\mathcal H)$-valued Calder\'on--Zygmund operator whose
associated kernel $K$ satisfies the size estimate
\begin{equation}\label{eq:size}
 \|K(x,y)\|_{B(\mathcal H)}
 \lesssim \frac{1}{|x-y|^d},
\end{equation}
and, whenever $2|x-x'|\leq |x-y|$, the $\alpha$-H\"older regularity estimate
\begin{align}\label{eq:smooth}
 \|K(x,y)-K(x',y)\|_{B(\mathcal H)}
 +\|K(y,x)-K(y,x')\|_{B(\mathcal H)}
 \lesssim
 \frac{|x-x'|^\alpha}{|x-y|^{d+\alpha}},
\end{align}
for some $0<\alpha\le1$.
The natural operator-valued analogue of the BMO condition in
the scalar case is
\begin{equation}\label{eq:operator-T1-cancellation}
 T1,\ T^*1\in \operatorname{BMO}^{rc}(\R^d;B(\mathcal H)).
\end{equation}
The corresponding weak boundedness property takes the form
\begin{equation}\label{eq:operator-weak-boundedness}
 \left|
 \left\langle
   T(\phi_Q\xi),\psi_Q\eta
 \right\rangle_{L_2(\R^d;\mathcal H)}
 \right|
 \lesssim
 |Q|\,\|\xi\|_{\mathcal H}\|\eta\|_{\mathcal H},
 \qquad
 \xi,\eta\in\mathcal H,
\end{equation}
uniformly over all cubes $Q\subset\R^d$, where
$\phi_Q,\psi_Q\in C_c^\infty(\R^d)$ are supported in $Q$ and satisfy
the standard normalization
\[
 \|\phi_Q\|_\infty,\|\psi_Q\|_\infty\leq 1,
 \qquad
 \|\nabla\phi_Q\|_\infty,\|\nabla\psi_Q\|_\infty
 \lesssim \ell(Q)^{-1}.
\]
An operator $T$ satisfying the above conditions will be called an
operator-valued $T1$ operator.

We next record the dyadic operators that will be used below in the
Hilbert-space setting on $\R^d$; see
\cite{HanninenHytonen2016,Hytonen2021}.
Fix a dyadic grid $\mathcal D$ in $\R^d$. For $I=I_1\times\cdots\times I_d\in\mathcal D$, let
$h_I^\eta=\bigotimes_{r=1}^d h_{I_r}^{\eta_r}$,
$\eta\in\{0,1\}^d\setminus\{0\}$, denote the normalized cancellative
Haar functions, where $h_J^0=|J|^{-1/2}\1_J$ and $h_J^1$ is the
normalized Haar function on the interval $J$. Write
\[
 \langle f,h_I^\eta\rangle=\int_I h_I^\eta(x)f(x)\,\dd x,
 \qquad
 f_I=\frac1{|I|}\int_I f(x)\,\dd x.
\]
An operator-valued cancellative dyadic shift of type
$(i,j)\in\N^2$ is given by
\begin{equation}\label{eq:operator-dyadic-shift}
 S_{\mathcal D}^{ij}f
 =\sum_{K\in\mathcal D}
   \sum_{\substack{I,J\in\mathcal D,\ I,J\subset K\\
       \ell(I)=2^{-i}\ell(K),\ \ell(J)=2^{-j}\ell(K)}}
   \sum_{\eta,\eta'\in\{0,1\}^d\setminus\{0\}}
   h_J^{\eta'}\,a_{I,J}^{K,\eta,\eta'}\,\langle f,h_I^\eta\rangle.
\end{equation}
The coefficients $a_{I,J}^{K,\eta,\eta'}\in B(\mathcal H)$ satisfy
\[
 \|a_{I,J}^{K,\eta,\eta'}\|_{B(\mathcal H)}
 \leq
 \frac{|I|^{1/2}|J|^{1/2}}{|K|}.
\]
Equivalently, the normalized coefficient family
\[
 \left\{
 \frac{|K|}{|I|^{1/2}|J|^{1/2}}\,a_{I,J}^{K,\eta,\eta'}
 \right\}_{K,I,J,\eta,\eta'}
\]
is uniformly bounded in $B(\mathcal H)$. Since every uniformly bounded
family of operators on a Hilbert space is automatically
$\mathcal R$-bounded, \cite[Theorem~1]{HanninenHytonen2016} yields
\begin{equation}\label{eq:operator-dyadic-shift-bound}
 \|S_{\mathcal D}^{ij}\|_{B(L_2(\R^d;\mathcal H))}
 \lesssim_d 1+\max\{i,j\},
\end{equation}

For a locally integrable $B(\mathcal H)$-valued symbol $b$, define
the dyadic paraproduct and its formal adjoint by
\begin{equation}\label{eq:operator-dyadic-paraproduct}
 \begin{aligned}
 \Pi_b^{\mathcal D}f
 &=\sum_{I\in\mathcal D}\sum_{\eta\in\{0,1\}^d\setminus\{0\}}
   h_I^\eta\,\langle b,h_I^\eta\rangle\,f_I,\\
 (\Pi_b^{\mathcal D})^*f
 &=\sum_{I\in\mathcal D}\sum_{\eta\in\{0,1\}^d\setminus\{0\}}
   \frac{\1_I}{|I|}
   \langle b,h_I^\eta\rangle^*\,\langle f,h_I^\eta\rangle.
 \end{aligned}
\end{equation}
These sums are initially understood for finite Haar expansions, or
equivalently through their pairings with suitable test functions.
When $\mathcal H=\C^N$, the matrix-valued paraproduct estimate yields
\begin{equation}\label{eq:operator-matrix-paraproduct-bound}
 \|\Pi_b^{\mathcal D}\|_{B(L_2(\R^d;\C^N))}
 \lesssim
 \log(N+1)\,\|b\|_{\mathrm{BMO}^c},
\end{equation}
and
\begin{equation}\label{eq:operator-matrix-paraproduct-bound*}
 \left\|\left(\Pi_{b^*}^{\mathcal D}\right)^*\right\|_{B(L_2(\R^d;\C^N))}
 \lesssim_d
 \log(N+1)\,\|b\|_{\mathrm{BMO}^r};
\end{equation}
see \cite{Katz1997}. The logarithmic dependence on $N$
is optimal; see \cite{NazarovPisierTreilVolberg2002}.

We now recall the dyadic representation theorem
\cite[Theorem~2]{HanninenHytonen2016}. Let $T$ be an operator-valued
$T1$ operator with kernel regularity exponent $\alpha\in(0,1]$, and assume that $T:L^p(\R^d;\mathcal H)\to L^p(\R^d;\mathcal H)$ is bounded.
Let $C_T$ denote the supremum of the constants appearing in its defining conditions. Then, for any $f,g\in C_c^1(\R^d;\mathcal H)$,
\begin{equation}\label{eq:operator-dyadic-representation}
 \begin{aligned}
 \langle Tf,g\rangle
 =\mathbb E_\omega\Bigg[
 4C_T\sum_{i,j\geq0}
  2^{-\frac{\alpha}{2}\max\{i,j\}}
  \langle S_{\mathcal D^\omega}^{ij}f,g\rangle +
 \left\langle
   \left(
    \Pi_{T1}^{\mathcal D^\omega}
    +(\Pi_{T^*1}^{\mathcal D^\omega})^*
   \right)f,g
 \right\rangle
 \Bigg].
 \end{aligned}
\end{equation}
Here we let $\mathcal D^0$ denote the standard dyadic grid in $\R^d$, and for
\[
 \omega=(\omega_k)_{k\in\mathbb Z}
 \in (\{0,1\}^d)^{\mathbb Z},
\]
we define
\[
 \mathcal D^\omega
 =
 \left\{
 Q+\sum_{\substack{k\in\mathbb Z\\2^{-k}<\ell(Q)}}
 2^{-k}\omega_k
 :
 Q\in\mathcal D^0
 \right\}.
\]

\subsection{The logarithmic upper bound}

For convenience, we write $T_j^\R$ and $T_j^\T$ for the
Fourier multipliers on $\R^d$ and $\T^d$ respectively, with the same symbol $m_j$.
We use the same notation \(T_j^\R\) and \(T_j^\T\) for their componentwise extensions to matrix-valued functions. For a $\mathbb M_N(\C)$-valued function $b$, we set
\[
 \|b\|_{W^{1,\infty}(\T^d;\mathbb M_N(\C))}
 :=
 \|b\|_{L^\infty(\T^d;\mathbb M_N(\C))}
 +
 \sum_{k=1}^d
 \|\partial_k b\|_{L^\infty(\T^d;\mathbb M_N(\C))}.
\]

\begin{lemma}\label{lem:periodic-calderon}
For every $N\geq1$ and
$b\in W^{1,\infty}(\T^d;\mathbb M_N(\C))$,
\begin{equation}\label{eq:periodic-calderon}
 \|[b,T_j^\T]\|_{B(L_2(\T^d;\C^N))}
 \leq C_d\log(N+1)\sum_{k=1}^d\|\partial_kb\|_\infty,
 \qquad 1\leq j\leq d.
\end{equation}
Viewing $b$ as acting on $S_2(\C^N)$ by left multiplication, the same estimate, with the same constant, holds on $L_2(\T^d;S_2(\C^N))$:
\begin{equation}\label{eq:periodic-calderon2}
 \|[b,T_j^\T]\|_{B(L_2(\T^d;S_2(\C^N)))}
 \leq C_d\log(N+1)\sum_{k=1}^d\|\partial_kb\|_\infty,
 \qquad 1\leq j\leq d.
\end{equation}
\end{lemma}

\begin{proof}
We first work on $\R$.
Let $H$ be the Hilbert transform with symbol
$-\mathrm i\operatorname{sgn}\xi$, so that
\(
 \left|\frac{\mathrm d}{\mathrm dx}\right|
 =H\frac{\mathrm d}{\mathrm dx}.
\)
We first record the kernel estimates and the weak boundedness property
for $\left[b,\left|\frac{\mathrm d}{\mathrm dx}\right|\right]$.
For $x\ne y$, the kernel is
\[
 K_b(x,y)=\frac{b(y)-b(x)}{\pi(x-y)^2}.
\]
The condition $b\in W^{1,\infty}(\R;\mathbb M_N(\C))$ gives
\begin{equation}\label{size}
 \|K_b(x,y)\|_{\mathbb M_N(\C)}
 \lesssim\left\|\frac{\mathrm d}{\mathrm dx}b\right\|_\infty
 \frac{1}{|x-y|},
\end{equation}
and, whenever $2|x-x'|\leq|x-y|$,
\begin{align*}
 \|K_b(x,y)-K_b(x',y)\|_{\mathbb M_N(\C)}
   +\|K_b(y,x)-K_b(y,x')\|_{\mathbb M_N(\C)}\lesssim
 \left\|\frac{\mathrm d}{\mathrm dx}b\right\|_\infty
 \frac{|x-x'|}{|x-y|^2}.
\end{align*}
Thus the kernel satisfies the standard size condition
\eqref{eq:size} and smoothness condition \eqref{eq:smooth},
with constants independent of $N$.

For the weak boundedness condition, take
$\phi_Q,\psi_Q\in C_c^\infty(\R)$ supported in $Q$, with the usual normalization
\[
 \|\phi_Q\|_\infty,\|\psi_Q\|_\infty\leq1,\qquad
 \|\phi_Q'\|_\infty,\|\psi_Q'\|_\infty\leq|Q|^{-1}.
\]
The identity $K_b(y,x)=-K_b(x,y)$ implies that, for
$\xi,\eta\in\C^N$,
\begin{align*}
 \left\langle \left[b,\left|\frac{\mathrm d}{\mathrm dx}\right|\right](\phi_Q\xi),\psi_Q\eta\right\rangle
 =\frac12\int_Q\int_Q
 \left\langle K_b(x,y)\xi,\eta\right\rangle\bigl(
 \phi_Q(y)\overline{\psi_Q(x)}
 -\phi_Q(x)\overline{\psi_Q(y)}
 \bigr)\,\mathrm dy\mathrm dx.
\end{align*}
Since
\[
 \frac12\bigl|
 \phi_Q(y)\overline{\psi_Q(x)}
 -\phi_Q(x)\overline{\psi_Q(y)}
 \bigr|\leq\frac{|x-y|}{|Q|},
\]
the kernel size estimate \eqref{size} yields
\begin{align*}
 \left|\left\langle \left[b,\left|\frac{\mathrm d}{\mathrm dx}\right|\right](\phi_Q\xi),\psi_Q\eta\right\rangle\right|
 &\lesssim
 \left\|\frac{\mathrm d}{\mathrm dx}b\right\|_\infty\frac{\|\xi\|_{\C^N}\|\eta\|_{\C^N}}{|Q|}
 \int_Q\int_Q\dd x\dd y\\
 &=\left\|\frac{\mathrm d}{\mathrm dx}b\right\|_\infty|Q|\,\|\xi\|_{\C^N}\|\eta\|_{\C^N}.
\end{align*}
This proves weak boundedness \eqref{eq:operator-weak-boundedness} with a constant independent of $N$.

For the BMO condition, note that we have the identity
\[
 \begin{aligned}
 \left[b,\left|\frac{\mathrm d}{\mathrm dx}\right|\right]1
 =-H\frac{\mathrm d}{\mathrm dx}b,\qquad
 \left[b,\left|\frac{\mathrm d}{\mathrm dx}\right|\right]^*1
 =H\left(\left(\frac{\mathrm d}{\mathrm dx}b\right)^*\right).
 \end{aligned}
\]
Using \cite[Theorem~6.4]{Mei2007}, we have
\begin{equation}\label{eq:rational-column-tests}
 \left\| \left[b,\left|\frac{\mathrm d}{\mathrm dx}\right|\right]1\right\|_{\mathrm{BMO}^{rc}}=\left\|H\frac{\mathrm d}{\mathrm dx}b\right\|_{\mathrm{BMO}^{rc}}
 \leq C\left\|\frac{\mathrm d}{\mathrm dx}b\right\|_\infty.
\end{equation}
Hence applying the scalar $T1$ theorem entrywise, we may deduce that for any fixed $N\in \N,$ $\left[b,\left|\frac{\mathrm d}{\mathrm dx}\right|\right]$ is $L_p(\R;\mathbb C^N)$-bounded.
Combining the dyadic representation \eqref{eq:operator-dyadic-representation} with the estimates
\eqref{eq:operator-dyadic-shift-bound},
\eqref{eq:operator-matrix-paraproduct-bound}, \eqref{eq:operator-matrix-paraproduct-bound*} and \eqref{eq:rational-column-tests}, we have
\begin{equation}\label{eq:rational-line-bound}
 \left\|\left[b,\left|\frac{\mathrm d}{\mathrm dx}\right|\right]\right\|_{B(L_2(\R;\C^N))}
 \leq C\log(N+1)\left\|\frac{\mathrm d}{\mathrm dx}b\right\|_\infty.
\end{equation}

Now we work on $\R^d$ for $d\geq2$. Let $\sigma$ be normalized surface measure on
$S^{d-1}$ and put $c_d=\int_{S^{d-1}}|\omega_1|\dd\sigma(\omega)$.
Rotational invariance gives
\begin{equation}\label{eq:rational-rotation}
 \frac{\xi_j^2}{|\xi|}
 =\frac1{c_d}\int_{S^{d-1}}
   \bigl((d+1)\omega_j^2-1\bigr)|\omega\cdot\xi|
   \dd\sigma(\omega).
\end{equation}
Consequently, we have \[
 [b,T_j^\R]
 =
 \frac1{c_d}
 \int_{S^{d-1}}
 \bigl((d+1)\omega_j^2-1\bigr)
 [b,|\partial_\omega|]
 \,\dd\sigma(\omega),
\]
where $\partial_\omega$ denotes the directional derivative in the
direction $\omega$. Fix $\omega\in S^{d-1}$ and write
$x=y+t\omega$, where $y\in\omega^\perp$ and $t\in\R$.
For each fixed $y$, set
\[
 f_{\omega,y}(t):=f(y+t\omega),
 \qquad
 b_{\omega,y}(t):=b(y+t\omega).
\]
Then, for $f\in L^2(\R^d;\C^N)$,
\[
 ([b,|\partial_\omega|]f)(y+t\omega)
 =
 \left(
 \left[b_{\omega,y},
       \left|\frac{\mathrm d}{\mathrm dt}\right|\right]
 f_{\omega,y}
 \right)(t).
\]
Applying \eqref{eq:rational-line-bound} on each line and
using Fubini's theorem, we obtain
\begin{align*}
 \|[b,|\partial_\omega|]\|_{B(L_2(\R^d;\C^N))}
 &\leq
 \sup_{y\in\omega^\perp}
 \left\|
 \left[b_{\omega,y},
       \left|\frac{\mathrm d}{\mathrm dt}\right|\right]
 \right\|_{B(L_2(\R;\C^N))}\\
 &\leq
 C\log(N+1)
 \sup_{y\in\omega^\perp}
 \left\|\frac{\mathrm d}{\mathrm dt}b_{\omega,y}\right\|_\infty\\
 &=C\log(N+1)\|\partial_\omega b\|_\infty\\
 &\leq
 C\log(N+1)\sum_{k=1}^d\|\partial_{x_k}b\|_\infty,
\end{align*}
where $C$ is independent of $\omega$ and $N$.

To pass from $\R^d$ to $\T^d$, we identify functions on $\T^d$ with their $(2\pi\Z)^d$-periodic
extensions to $\R^d$. We first take $b$ and $f$ to be matrix-valued and
$\C^N$-valued trigonometric polynomials on $\T^d$. Choose $\varphi\in C_c^\infty(\R^d)$ with
$\|\varphi\|_2=1$ and set
$\varphi_R(x):=R^{-d/2}\varphi(x/R)$.
By the fact that $m_j$ has Lipschitz constant at most $3$,
\begin{align}\label{Lip error}
    \|T_j^\R(\varphi_Rg)
       -\varphi_RT_j^\T g\|_2
 \leq3R^{-1}\|\nabla\varphi\|_2\sum_{n\in\Z^d}\|\widehat g(n)\|, \qquad \forall g\in L_2(\T^d;\mathbb M_N(\C)).
\end{align}
Take $\chi\in C_c^\infty(\R^d)$ with $0\leq\chi\leq1$ and
$\chi=1$ on $\operatorname{supp}\varphi$, and put
$b_R(x)=\chi(x/R)b(x)$. Applying the identity $b_R\varphi_R=b\varphi_R$ and \eqref{Lip error} with $g=f$ and $g=bf,$ we have
\[
 [b_R,T_j^\R](\varphi_Rf)
 =\varphi_R[b,T_j^\T]f+R^{-1}\mathcal E,
\]
where we have \[
 \|\mathcal E\|_2
 \leq
 3\|\nabla\varphi\|_2
 \left(
 \|b\|_\infty\sum_{n\in\Z^d}\|\widehat f(n)\|
 +\sum_{n\in\Z^d}\|\widehat{bf}(n)\|
 \right).
\]
Moreover, we have
\[
 \sum_{k=1}^d\|\partial_kb_R\|_\infty
 \leq\sum_{k=1}^d\|\partial_kb\|_\infty
       +O(R^{-1})\|b\|_\infty.
\]
Let $R\to\infty$,
and we obtain \eqref{eq:periodic-calderon} for the matrix-valued trigonometric polynomial $b$ and $\C^N$-valued trigonometric polynomial $f$.
The general case follows by a density argument.
\end{proof}

Now we prove \eqref{eq:rational-torus-bound}. Choose an integer $q\geq1$ for which all entries of $q\theta$ are integers. On the $q^d$-dimensional Hilbert space
$\ell_2(\Z_q^d)$, with standard basis $\{\varepsilon_r\}_{r\in\Z_q^d}$, define
\begin{equation}\label{eq:rational-matrix-unitaries}
 V_j\varepsilon_r
 =\exp\left(2\pi\mathrm i\sum_{k<j}\theta_{jk}r_k\right)
       \varepsilon_{r+e_j},
 \qquad 1\leq j\leq d.
\end{equation}
Here $e_j$ denotes the $j$th standard basis vector of $\Z^d$.
The condition $q\theta_{jk}\in\Z$ ensures that
\eqref{eq:rational-matrix-unitaries} is independent of the choice of
representatives for $r_k\in\Z_q$. The operators $V_j$ are unitary, and direct calculation
gives
\[
 V_jV_k=e^{2\pi\mathrm i\theta_{jk}}V_kV_j.
\]
Consider the transference map
\begin{equation}\label{eq:rational-embedding}
 \pi:\T_\theta^d\longrightarrow
       L_\infty(\T^d)\,\overline\otimes\,\mathbb M_{q^d}(\C),\qquad
 \pi(U^n)(x)=e^{\mathrm i n\cdot x}V^n,
 \qquad \forall n\in\Z^d.
\end{equation}
It extends to a trace-preserving normal unital $*$-homomorphism, where $L_\infty(\T^d)\,\overline\otimes\,\mathbb M_{q^d}(\C)$ is equipped with the trace
$\int_{\T^d}\otimes\operatorname{tr}_{q^d}$, and $\operatorname{tr}_{q^d}$ is the normalized matrix trace.
Moreover, it intertwines the partial derivatives and Fourier
multipliers in the two settings:
\begin{equation}\label{eq:rational-multiplier-intertwining}
 \partial_{x_k}\pi(b)=\pi(\partial_kb),\qquad
 T_j^\T\pi(f)=\pi(T_jf).
\end{equation}
Thus
\[
 \pi([b,T_j]f)=[\pi(b),T_j^\T]\pi(f).
\]
By Lemma~\ref{lem:periodic-calderon} with $N=q^d$,
\begin{equation}\label{eq:rational-matrix-bound}
 \begin{aligned}
 \|[\pi(b),T_j^\T]\|_{B(L_2(\T^d;S_2(\C^{q^d})))}
 &\leq C_d\log(1+q^d)\sum_{k=1}^d\|\partial_kb\|_\infty\\
 &\leq C_d\log(1+q)\sum_{k=1}^d\|\partial_kb\|_\infty.
 \end{aligned}
\end{equation}
Restricting this estimate to the image of $\pi$ proves
\eqref{eq:rational-torus-bound}.
\subsection{Optimality}
For every integer $q\geq1$, put
\[
 \theta_q=\frac1q
 \begin{pmatrix}0&1\\-1&0\end{pmatrix}\oplus0_{d-2}.
\]
Now it suffices to show that there exists a self-adjoint $b_q\in C^\infty(\T_{\theta_q}^d)$ such that
\begin{equation}\label{eq:rational-torus-sharpness}
 \|[b_q,T_1]\|_{B(L_2(\T_{\theta_q}^d))}
 \geq c\log(1+q)\sum_{k=1}^d\|\partial_kb_q\|_\infty.
\end{equation}
Here $c>0$ is an absolute constant.

\begin{proof}[Proof of optimality]
We first take $d=2$ and
\[
 \theta_q=\frac1q\begin{pmatrix}0&1\\-1&0\end{pmatrix}.
\]
On $\ell_2(\Z_q)$ with standard basis $\{\varepsilon_r\}_{r\in\Z_q}$, put
\[
\mathcal V_1\varepsilon_r=e^{2\pi\mathrm ir/q}\varepsilon_r,\qquad
\mathcal V_2\varepsilon_r=\varepsilon_{r+1\bmod q}.
\]
Then $\mathcal V_1\mathcal V_2=e^{2\pi\mathrm i/q}\mathcal V_2\mathcal V_1$, and the same construction
as in \eqref{eq:rational-embedding} gives another transference map
\[
  \pi:\T_{\theta_q}^2\longrightarrow
 L_\infty(\T^2)\,\overline\otimes\,\mathbb M_q(\C),\qquad \pi(U^n)(x)=e^{\mathrm i n\cdot x}\mathcal V^n,\qquad \forall n\in\Z^2.
\]
This transference preserves the trace
$\int_{\T^2}\otimes\operatorname{tr}_q$, the derivatives, and the
Fourier multipliers in both settings.
Let
\[
 \mathcal H:=((2\pi/q)\Z/(2\pi\Z))^2,\qquad
 \beta_{(2\pi r/q,2\pi s/q)}:=\operatorname{Ad}(\mathcal V_1^s\mathcal V_2^{-r}).
\]
We can show that a smooth matrix-valued function $b$ belongs to the image of $\pi$ if and
only if
\begin{equation}\label{eq:rational-covariance}
 b(x+h)=\beta_h(b(x)),\qquad h\in\mathcal H.
\end{equation}
Indeed, the matrices $\mathcal V^{(r',s')}=\mathcal V_1^{r'}\mathcal V_2^{s'}$, $0\leq r',s'<q$, form a joint
eigenbasis for the action $\beta$, with
$\beta_h(\mathcal V^n)=e^{\mathrm i n\cdot h}\mathcal V^n$.
Thus \eqref{eq:rational-covariance} is equivalent to $\widehat b(n)=\lambda_n \mathcal V^n$ for some $\lambda_n\in\C$, and hence equivalent to $$b=\pi\left(\sum_{n\in\Z^2}\lambda_nU^n\right).$$ 

For $q=1,2,3$, take $b=(U_1+U_1^*)/2$.
Then
\[
 \|[b,T_1]\|_{B(L_2(\T_{\theta_q}^2))}\geq\|[b,T_1]1\|_2=\|b\|_2=1/\sqrt2,\qquad
 \sum_{k=1}^2\|\partial_kb\|_\infty=1,
\]
which gives
\eqref{eq:rational-torus-sharpness}.
Hence we may assume that $q\geq4$. Put $N=\lfloor q/2\rfloor$.
Let $D_N,E_N$ be the matrices of Proposition~\ref{prop:finite-obstruction}, and set
\[
 S=D_N\oplus0_{q-2N},\qquad A=E_N\oplus0_{q-2N}.
\]
Then it follows from the definitions of $D_N,E_N$ that $S$ has integer spectrum, $A=A^*$, and
\begin{equation}\label{eq:rational-sharp-matrices}
 0<\|[S,A]\|\leq\pi,\qquad
 \|[A,|S|]\|\geq\kappa_N.
\end{equation}
Choose a unit vector $v\in\ell_2(\Z_q)$ with
\begin{equation}\label{eq:rational-maximizing-vector}
 \|v\|=1,\qquad
 \|[A,|S|]v\|=\|[A,|S|]\|.
\end{equation}

Now we fix a unit vector $w\in\ell_2(\Z_q)$ and a ball $O\subset\T^2$ whose $\mathcal H$-translates have
disjoint closures. Take real smooth functions $\chi,\psi$ supported
in $O$, with $0\leq\chi\leq1$, $\psi\ne0$, and $\chi=1$ on
$\operatorname{supp}\psi$. For positive integers $M$ divisible by $q$,
put $E_M(x)=e^{\mathrm ix_1MS}$ and define
\begin{align*}
 B_M(x)&:=\sum_{h\in\mathcal H}\chi(x-h)
       \beta_h\bigl(E_M(x)AE_M(x)^*\bigr),\\
 F_{M,z}(x)&:=\sum_{h\in\mathcal H}\psi(x-h)
       \beta_h\bigl(E_M(x)z\otimes w^*\bigr),\qquad z\in\ell_2(\Z_q).
\end{align*}
Since $q\mid M$ and $S$ has integer spectrum,
$E_M(x+h)=E_M(x)$ for every $h\in\mathcal H$.
Hence both $B_M$ and $F_{M,z}$ satisfy
\eqref{eq:rational-covariance}, which implies the existence of smooth elements $b_{M,q},f_{M,z}$ with
$B_M=\pi(b_{M,q}),F_{M,z}=\pi(f_{M,z})$. 

Now we show that $b_{M,q}$ is the desired symbol. Note that we have $b_{M,q}=b_{M,q}^*$ and
\begin{equation}\label{eq:rational-sharp-derivatives}
 \begin{aligned}
     \sum_{k=1}^2\|\partial_kb_{M,q}\|_\infty
      =\sum_{k=1}^2\|\partial_{x_k}B_M\|_\infty
 &\le\sum_{k=1}^2\left\|\sum_{h\in\mathcal H}\partial_{x_k}\chi(x-h)
       \beta_h\bigl(E_M(x)AE_M(x)^*\bigr)\right\|_\infty\\
&\qquad+\left\|\sum_{h\in\mathcal H}\chi(x-h)
\beta_h\bigl(E_M(x)[\mathrm i MS,A]E_M(x)^*\bigr)\right\|_\infty\\
 &\leq \|A\|\sum_{k=1}^2\|\partial_{x_k}\chi\|_\infty+M\|[S,A]\|.
 \end{aligned}
\end{equation}
Moreover, $\chi\psi=\psi$ and $|\mathcal H|=q^2$ imply
\begin{equation}\label{eq:rational-test-identities}
 B_MF_{M,z}=F_{M,Az},\qquad
 \|F_{M,z}\|_2=\sqrt q\,\|\psi\|_2\|z\|.
\end{equation}
Using the fact that $m_1$ has Lipschitz constant at most $3$,
\begin{align*}
 \left\|T_1^\T(\psi g)-\psi T_1^\T g\right\|_2^2
 &=
 \sum_{k\in\Z^2}
 \left\|
 \sum_{n\in\Z^2}
 \bigl(m_1(k)-m_1(n)\bigr)
 \widehat\psi(k-n)\widehat g(n)
 \right\|_{L_2(\mathbb M_q(\C),\operatorname{tr}_q)}^2 \\
 &\leq
 9
 \left(
 \sum_{\ell\in\Z^2}
 |\ell|\,|\widehat\psi(\ell)|
 \right)^2
 \|g\|_2^2.
\end{align*}
Applying the preceding estimate with $\psi(\cdot-h)$ and $$g_h(x):=\beta_h\bigl(E_M(x)z\otimes w^*\bigr),$$ we have
\[
 \frac1{M}T_1^\T F_{M,z}
 =F_{M,|S|z}+\frac1M\mathcal E_{M,z},
 \qquad
 \|\mathcal E_{M,z}\|_2\leq C\|z\|,
\]
where $C$ depends only on the fixed $q,S,\psi$.
Applying the identity $B_MF_{M,z}=F_{M,Az}$ and this estimate with
$z=v$ and $z=Av$, we obtain
\begin{equation}\label{eq:rational-commutator-test}
 \frac1{M}[B_M,T_1^\T]F_{M,v}
 =F_{M,[A,|S|]v}+\frac1M\mathcal E_M,
\end{equation}
where
\begin{equation}\label{eq:rational-commutator-error}
 \mathcal E_M=B_M\mathcal E_{M,v}-\mathcal E_{M,Av},
 \qquad
 \|\mathcal E_M\|_2\leq2C\|A\|,
\end{equation}
since $\|B_M\|_\infty\leq\|A\|$.
Using the isometry $\pi$, the identity
$\|F_{M,z}\|_2=\sqrt q\,\|\psi\|_2\|z\|$, and the choice of $v$,
we have
\begin{eqnarray*}
 \|[b_{M,q},T_1]\|_{B(L_2(\T_{\theta_q}^2))}
 &\overset{\eqref{eq:rational-test-identities}}{\geq}&
 \frac{\|[b_{M,q},T_1]f_{M,v}\|_2}{\|f_{M,v}\|_2}
 \\[6pt]
 &\overset{\eqref{eq:rational-multiplier-intertwining}}{=}&
 \frac{\|[B_M,T_1^\T]F_{M,v}\|_2}{\|F_{M,v}\|_2}
 \\[6pt]
 &\overset{\eqref{eq:rational-commutator-test}}{=}&
 \frac{\|M F_{M,[A,|S|]v}+\mathcal E_M\|_2}
      {\|F_{M,v}\|_2}
 \\[6pt]
 &\overset{\eqref{eq:rational-commutator-error}}{\ge}&
 M\|[A,|S|]v\|_{\ell_2(\Z_q)}
 -\frac{2C\|A\|}{\sqrt q\,\|\psi\|_2}.
 \\[6pt]
 &\overset{\eqref{eq:rational-maximizing-vector}}{\ge}&
 M\|[A,|S|]\|
 -\frac{2C\|A\|}{\sqrt q\,\|\psi\|_2}.
\end{eqnarray*}
Since $q$, $A$, and $\psi\ne0$ are fixed, enlarging $C$ gives
\[
 \|[b_{M,q},T_1]\|_{B(L_2(\T_{\theta_q}^2))}
 \geq M\|[A,|S|]\|-C,
\]
where the constant is independent of $M$.
Combining this with \eqref{eq:rational-sharp-derivatives}
and letting $M\to\infty$ through multiples of $q$, we obtain
\[
 \liminf_{\substack{M\to\infty\\q\mid M}}
 \frac{\|[b_{M,q},T_1]\|_{B(L_2(\T_{\theta_q}^2))}}
      {\sum_{k=1}^2\|\partial_kb_{M,q}\|_\infty}
 \geq\frac{\|[A,|S|]\|}{\|[S,A]\|}
 \geq\frac{\kappa_N}{\pi}.
\]
Here $q$ and the cutoffs are fixed before $M$ tends to infinity.
Since $\kappa_N\gtrsim\log(N+1)$ and $N=\lfloor q/2\rfloor$,
we may choose a sufficiently large finite $M$ to obtain a lower bound
$c\log(1+q)$ for this ratio. Then $b_{M,q}$ is the desired symbol.

For the general case $d\geq 2$, it suffices to consider
\[
 b_q=b_{M,q}\otimes 1^{\otimes(d-2)},
\]
where $b_{M,q}$ is the two-dimensional symbol attaining the logarithmic lower bound.
\end{proof}

\begin{remark}
Theorem~\ref{thm:torus} also follows from the sharpness assertion
in Theorem~\ref{thm:torus-rational} by an approximation argument.
Set $J:=\begin{pmatrix}
    0&1\\ -1&0
\end{pmatrix}$ and
\[
 C(\alpha):=
 \sup_{0\ne b\in W^{1,\infty}(\T_{\alpha J}^2)}
 \frac{\|[b,T_1]\|_{B(L_2(\T_{\alpha J}^2))}}
      {\|b\|_{W^{1,\infty}}},
\]
The trace-preserving embedding
\[
 \Pi_k:\T_{k^2\alpha J}^2\longrightarrow\T_{\alpha J}^2,
 \qquad U\mapsto U^k,\quad V\mapsto V^k,
\]
scales both the derivatives and $T_1$ by $k$, and therefore
$C(k^2\alpha)\leq C(\alpha)$.
By \cite[Theorem~4.2]{Gao2018}, the function $\alpha\mapsto C(\alpha)$ is
lower semicontinuous.
Hence, for irrational $\alpha$ and every positive integer $q$, we may
choose a sequence $\{k_n\}$ such that $k_n^2\alpha\bmod 1\to 1/q$, which
implies that
\[
 C(\alpha)\geq\limsup_n C(k_n^2\alpha)\geq C(1/q)\gtrsim\log(1+q)
 \qquad(q\geq1),
\]
and therefore $C(\alpha)=\infty$. By the closed graph theorem, we show the existence of $b.$ The general case follows similarly and we omit the details.
\end{remark}

\section{Calder\'on commutators and the limits of \texorpdfstring{$T1$}{T1} testing}
\label{sec:t1}

The estimate in Lemma~\ref{lem:periodic-calderon} gives a
finite-dimensional version of the operator-valued Calder\'on theorem, with a logarithmic loss depending on matrix dimension. However, this estimate does not directly yield a bound for $B(\ell_2)$-valued symbols. On the other hand, the failure of an unrestricted operator-valued $T1$
theorem does not, by itself, settle the corresponding question for Calder\'on commutators, which possess additional structure. It is therefore natural to ask whether $b\in W^{1,\infty}(\R;B(\ell_2))$ still
implies that $\left[b,\left|\frac{\mathrm d}{\mathrm dx}\right|\right]$
is bounded on $L_2(\R;\ell_2)$.
Theorem~\ref{thm:operator-valued-Calderon} gives a negative answer. The proof also relies on the finite obstruction constructed in Section \ref{s2}, and we give it in the next subsection.

\subsection{An operator-valued first commutator with bounded \texorpdfstring{$T1$}{T1} tests}
We now prove Theorem~\ref{thm:operator-valued-Calderon} in the
one-dimensional case $d=1$. The higher-dimensional
case follows by a straightforward extension of the same argument, and
we omit the details. Let
\(
 b:\R\to B(\ell_2)
\).
We use the norm $$\|b\|_{W^{1,\infty}(\R;B(\ell_2))}:= \|b\|_{L^\infty(\R;B(\ell_2))}+\left\|\frac{\mathrm d}{\mathrm dx}b\right\|_{L^\infty(\R;B(\ell_2))}.$$ Consistently with the notation used in the quantum
setting, we define Calder\'on's first commutator by
\(
 \left[b,\left|\frac{\mathrm d}{\mathrm dx}\right|\right].
\)

\begin{proof}[Proof of Theorem \ref{thm:operator-valued-Calderon}]
Let $\mathsf D_n,\mathsf A_n\in\mathbb M_{2N_n}(\C)$ be the matrices
given by Corollary~\ref{cor:scaled-obstruction}. Recall that
\begin{equation}\label{eq:operator-block-estimates}
 \|\mathsf A_n\|\leq\frac{2^{-n}}6,
 \qquad
 \|[\mathsf D_n,\mathsf A_n]\|\leq2^{-n},
 \qquad
 \|[|\mathsf D_n|,\mathsf A_n]\|\geq n.
\end{equation}
Identify
\(
 \ell_2=\bigoplus_{n\geq1}\C^{2N_n},
\)
and define
\begin{equation}\label{eq:operator-b}
 \begin{aligned}
 b_n(x)
 &:=
 \bigoplus_{k=1}^n
 e^{\mathrm ix\mathsf D_k}\mathsf A_k
 e^{-\mathrm ix\mathsf D_k},\\
 b(x)
 &:=
 \bigoplus_{n\geq1}
 e^{\mathrm ix\mathsf D_n}\mathsf A_n
 e^{-\mathrm ix\mathsf D_n}.
 \end{aligned}
\end{equation}
By \eqref{eq:operator-block-estimates},
\[
 \sup_{x\in\R}\|b(x)-b_n(x)\|
 =
 \sup_{k>n}\|\mathsf A_k\|
 \longrightarrow0.
\]
Moreover,
\begin{equation}\label{eq:operator-b-derivative}
 \begin{aligned}
 \frac{\mathrm d}{\mathrm dx}b_n(x)
 &=
 \bigoplus_{k=1}^n
 \mathrm i e^{\mathrm ix\mathsf D_k}
 [\mathsf D_k,\mathsf A_k]
 e^{-\mathrm ix\mathsf D_k},\\
 \frac{\mathrm d}{\mathrm dx}b(x)
 &=
 \bigoplus_{n\geq1}
 \mathrm i e^{\mathrm ix\mathsf D_n}
 [\mathsf D_n,\mathsf A_n]
 e^{-\mathrm ix\mathsf D_n}.
 \end{aligned}
\end{equation}
Again by \eqref{eq:operator-block-estimates},
\[
 \sup_{x\in\R}
 \left\|
 \frac{\mathrm d}{\mathrm dx}b(x)
 -
 \frac{\mathrm d}{\mathrm dx}b_n(x)
 \right\|
 =
 \sup_{k>n}\|[\mathsf D_k,\mathsf A_k]\|
 \longrightarrow0.
\]
Consequently,
\(
 b\in W^{1,\infty}(\R;B(\ell_2)).
\)

We next compute $\left|\frac{\mathrm d}{\mathrm dx}\right|b$. Let
$G_{N_n}$, $D_{N_n}$, $E_{N_n}$, and $X_{N_n}$ be defined as in
Section~\ref{s2}. A direct calculation gives
\[
e^{\mathrm ixD_{N_n}}E_{N_n}e^{-\mathrm ixD_{N_n}}
=
\begin{pmatrix}
0 &
\bigl((G_{N_n})_{rs}e^{\mathrm ix(2^r+2^s)}\bigr)_{r,s=1}^{N_n}
\\[4pt]
\bigl((G_{N_n}^*)_{rs}e^{-\mathrm ix(2^r+2^s)}\bigr)_{r,s=1}^{N_n}
&
0
\end{pmatrix}.
\]
Since $2^r+2^s>0$, applying
$\left|\frac{\mathrm d}{\mathrm dx}\right|$ entrywise and using the
definition of $\mathsf D_n$ yields
\begin{equation}\label{eq:operator-Db}
\begin{aligned}
 \left|\frac{\mathrm d}{\mathrm dx}\right|b_n(x)
 &=
 \bigoplus_{k=1}^n
 e^{\mathrm ix\mathsf D_k}\mathsf B_k
 e^{-\mathrm ix\mathsf D_k},\\
 \left|\frac{\mathrm d}{\mathrm dx}\right|b(x)
 &=
 \bigoplus_{n\geq1}
 e^{\mathrm ix\mathsf D_n}\mathsf B_n
 e^{-\mathrm ix\mathsf D_n}.
\end{aligned}
\end{equation}
where we set $$\mathsf B_n
:=
 \frac{2^{-n}}{\pi}
 \begin{pmatrix}
 0&X_{N_n}\\
 X_{N_n}^*&0
 \end{pmatrix}.$$
Recall from the proof of Proposition~\ref{prop:finite-obstruction} that
$\|X_{N_n}\|\leq\pi$.
Therefore  
\[
 \sup_{x\in\R}
 \left\|
 \left|\frac{\mathrm d}{\mathrm dx}\right|b(x)
 -
 \left|\frac{\mathrm d}{\mathrm dx}\right|b_n(x)
 \right\|
 =\sup_{k>n}\|\mathsf B_k\|\leq2^{-(n+1)}
 \longrightarrow0.
\]
and consequently
\(
 \left|\frac{\mathrm d}{\mathrm dx}\right|b
 \in L_\infty(\R;B(\ell_2)).
\)

To show unboundedness, fix $\varphi\in\mathcal S(\R)$ with
$\|\varphi\|_2=1$ and $\operatorname{supp}\widehat\varphi\subset[-1,1]$.
Choose a unit vector $v_n\in\C^{2N_n}$ such that
\[
 \|[\mathsf A_n,|\mathsf D_n|]v_n\|_{\C^{2N_n}}
 =\|[\mathsf A_n,|\mathsf D_n|]\|.
\]
Define
$f_n\in\mathcal S(\R;\ell_2)$ by
\[
 f_n(x)
 :=
 0\oplus\cdots\oplus
 \underbrace{e^{\mathrm ix\mathsf D_n}\varphi(x)v_n}_{n\text{-th block}}
 \oplus\cdots .
\]
Since
$e^{\mathrm ix\mathsf D_n}$ is unitary,
\[
 \|f_n\|_{L_2(\R;\ell_2)}
 =
 \|\varphi\|_2\|v_n\|
 =1.
\] Taking the Fourier transform gives
\begin{equation}\label{eq:operator-fiber}
 \mathcal F_x\left(
 e^{-\mathrm ix\mathsf D_n}
 \left[b,\left|\frac{\mathrm d}{\mathrm dx}\right|\right]f_n
 \right)(\xi)
 =
 [\mathsf A_n,|\xi\mathbf{Id}_{2N_n}+\mathsf D_n|]\widehat\varphi(\xi)v_n.
\end{equation}
The elementary estimate
$\||\xi\mathbf{Id}_{2N_n}+\mathsf D_n|-|\mathsf D_n|\|\leq|\xi|$ yields
\[
 \left\|
 \left[b,\left|\frac{\mathrm d}{\mathrm dx}\right|\right]f_n
 \right\|_2
 \geq
 \|[\mathsf A_n,|\mathsf D_n|]v_n\|
 -2\|\mathsf A_n\|
 \geq
 n-\frac{2^{-n}}3.
\]
Thus $\left[b,\left|\frac{\mathrm d}{\mathrm dx}\right|\right]$ is
unbounded.

The kernel estimates and the weak boundedness calculation are the
same as those in the proof of Lemma~\ref{lem:periodic-calderon}.
These estimates depend only on the Lipschitz norm of the symbol and
are independent of the matrix dimension. Applying them to the finite
block sums in \eqref{eq:operator-b} and passing to the limit using
\eqref{eq:operator-b-derivative} verifies
\eqref{eq:size}, \eqref{eq:smooth} with $d=1$ and $\alpha=1$, and
\eqref{eq:operator-weak-boundedness}.
Since $b=b^*$, we have
\begin{equation}\label{nsymm}
    \left[b,\left|\frac{\mathrm d}{\mathrm dx}\right|\right]1=-\left|\frac{\mathrm d}{\mathrm dx}\right|b,
 \qquad
 \left[b,\left|\frac{\mathrm d}{\mathrm dx}\right|\right]^*1=\left|\frac{\mathrm d}{\mathrm dx}\right|b
\end{equation}
in the distributional sense. By \eqref{eq:operator-Db}, both testing functions
belong to $L^\infty(\R;B(\ell_2))$. Since
$L^\infty(\R;B(\ell_2))\subset\operatorname{BMO}^{rc}(\R;B(\ell_2))$,
this also verifies \eqref{eq:operator-T1-cancellation}.
\end{proof}

\begin{remark}\label{rem1}
Theorem~\ref{thm:operator-valued-Calderon} shows that an operator-valued
$T1$ operator need not be $L_2$-bounded, even when
$T1,T^*1\in L_\infty(\R;B(\ell_2))$.
\end{remark}
\begin{remark}
    Theorem~\ref{thm:operator-valued-Calderon} does not contradict the
main result of Hong, Liu, and Mei~\cite{HongLiuMei2020}, since \eqref{nsymm} implies that $ \left[b,\left|\frac{\mathrm d}{\mathrm dx}\right|\right]$ does not satisfy the symmetric condition $(T1)^*=T^*1.$
\end{remark}

% \begin{corollary}\label{cor:failure-naive-operator-T1}
% Standard operator-norm Calder\'on--Zygmund kernel estimates, weak
% boundedness, and the tests $T1,T^*1\in\mathrm{BMO}(B(H))$ do not
% imply boundedness on $L_2(\R;H)$ for arbitrary operator-valued
% singular integrals. The conclusion remains false if the tests are
% strengthened to $T1,T^*1\in L_\infty(\R;B(H))$.
% \end{corollary}

% \begin{proof}
% All of the stated hypotheses hold for the operator in
% Theorem~\ref{thm:operator-valued-Calderon}, whereas its $L_2$ norm
% is infinite.
% \end{proof}

\subsection{Quantum BMO and a prospective \texorpdfstring{$T1$}{T1} principle}
In this subsection, we introduce the BMO spaces on quantum Euclidean
spaces and discuss the corresponding $T1$ theory. The canonical
translation action $\alpha_x$ on $\R_\theta^d$ is determined by
\[
 \alpha_x(\lambda_\theta(\xi))
 =
 e^{\mathrm i x\cdot\xi}\lambda_\theta(\xi),
 \qquad x,\xi\in\R^d.
\]
For $a\in\R_\theta^d$, the quantum BMO seminorm is defined as
\begin{align}
 \|a\|_{\mathrm{BMO}^c(\R_\theta^d)}
 &=
 \sup_Q
 \left\|
 \frac1{|Q|}
 \int_Q
 \bigl(\alpha_x(a)-(\alpha(a))_Q\bigr)^*
 \bigl(\alpha_x(a)-(\alpha(a))_Q\bigr)
 \,\dd x
 \right\|^{1/2},
 \label{eq:quantum-BMO-c}\\
 \|a\|_{\mathrm{BMO}^r(\R_\theta^d)}
 &=
 \|a^*\|_{\mathrm{BMO}^c(\R_\theta^d)},
 \qquad
 \|a\|_{\mathrm{BMO}^{rc}(\R_\theta^d)}
 =
 \max\bigl\{
   \|a\|_{\mathrm{BMO}^c(\R_\theta^d)},
   \|a\|_{\mathrm{BMO}^r(\R_\theta^d)}
 \bigr\}.
 \label{eq:quantum-BMO-r}
\end{align}
Here the supremum is taken over all cubes $Q\subset\R^d$, and $$(\alpha(a))_Q
 =
 \frac1{|Q|}
 \int_Q\alpha_x(a)\,\dd x.$$ 
One may equivalently describe these spaces in terms of the quantum
heat semigroup, using the corresponding row and column
square-oscillation norms. These BMO spaces, together with a theory of
singular integral operators on quantum Euclidean spaces, were developed
by Gonz\'alez-P\'erez, Junge, and Parcet~\cite{GJP}. Their framework
allows for nonconvolution kernels and separates the a priori
$L_2$-boundedness assumption from the kernel conditions used to derive
endpoint and $L_p$ estimates. In particular, the quantum Riesz
transforms satisfy the endpoint estimate
\begin{equation}\label{eq:quantum-Riesz-BMO}
 \|R_j a\|_{\mathrm{BMO}^{rc}(\R_\theta^d)}
 \leq C_d\|a\|_\infty.
\end{equation}

Theorem~\ref{thm:euclidean} imposes a concrete restriction on any quantum $T1$
principle. For $b=b^*\in W^{1,\infty}(\R_\theta^d)$, we have, in the
distributional sense,
\begin{equation}\label{eq:quantum-commutator-tests}
[b,T_j]1=-T_jb=R_j\partial_jb,
 \qquad
 [b,T_j]^*1=T_jb.
\end{equation}
These identities are obtained in exactly the same way as in the
operator-valued setting. Therefore, by \eqref{eq:quantum-commutator-tests} and
\eqref{eq:quantum-Riesz-BMO},
\begin{equation}\label{eq:quantum-commutator-BMO-tests}
 \|[b,T_j]1\|_{\mathrm{BMO}^{rc}}
 +
 \|[b,T_j]^*1\|_{\mathrm{BMO}^{rc}}
 \leq
 2C_d\|\partial_j b\|_\infty.
\end{equation}
Hence the symbols constructed in Theorem~\ref{thm:euclidean} satisfy
the natural quantum BMO testing conditions, even though the
corresponding commutators are unbounded on $L_2(\R_\theta^d)$.

We recall the kernel framework of
\cite[Sections~1.2.1 and~2.3.3]{GJP} in our Fourier normalization.
For a quantum kernel
$k\in\left(
\mathcal S(\R_\theta^d)\widehat\otimes
\mathcal S(\R_\theta^d)^{\mathrm{op}}
\right)'$,
the associated singular integral operator is formally given by
\begin{equation}\label{eq:gjp-kernel-operator}
 T_k a=(\mathrm{id}\otimes\tau_\theta)
       \bigl(k\bullet(1\otimes a)\bigr),
\end{equation}
with singular kernel $k$ interpreted by duality and
$(a\otimes b)\bullet(c\otimes e)=ac\otimes eb$.
The normal $*$-homomorphism
\[
 \pi_\theta:L_\infty(\R^d)\longrightarrow
 \R_\theta^d\overline\otimes(\R_\theta^d)^{\mathrm{op}},\qquad
 \pi_\theta(e^{\mathrm i\xi\cdot(\cdot)})
 =\lambda_\theta(\xi)\otimes\lambda_\theta(\xi)^*
\]
defines the quantum distance $d_\theta=\pi_\theta(|\cdot|)$
by spectral calculus. At $\theta=0$, we have
$\pi_0(f)(x,y)=f(x-y)$ and $d_0(x,y)=|x-y|$.

Let $ \mathcal L(\mathbb F_d)$ denote the group von Neumann algebra of the free group $\mathbb F_d.$ For a free semicircular family $(s_j)_{j=1}^d\subset \mathcal L(\mathbb F_d)$, define the quantum gradient by
\[
 \nabla_\theta a=\sum_{j=1}^d s_j\otimes\partial_j a.
\]
\cite[Theorem A]{GJP} states that $T_k:L_p(\R_\theta^d)\to L_p(\R_\theta^d),1<p<\infty$ is bounded if $T_k:L_2(\R_\theta^d)\to L_2(\R_\theta^d)$
is bounded, $k$ satisfies the quantum smoothness condition
\begin{equation}\label{eq:gjp-kernel-smoothness}
 \begin{aligned}
 \bigl\|d_\theta^\rho\bullet
     (\nabla_\theta\otimes\mathrm{id})k
     \bullet d_\theta^{d+1-\rho}\bigr\|
 +
 \bigl\|d_\theta^\rho\bullet
     (\mathrm{id}\otimes\nabla_\theta)k
     \bullet d_\theta^{d+1-\rho}\bigr\|
 \leq C
 \end{aligned}
\end{equation}
for $\rho\in\left\{0,\frac{d+1}{2},d+1\right\},$ and  $(\nabla_\theta\otimes\mathrm{id})k,(\mathrm{id}\otimes\nabla_\theta)k$ are affiliated with $\mathcal L(\mathbb F_d)\overline\otimes\R_\theta^d
\overline\otimes(\R_\theta^d)^{\mathrm{op}}$.
At $\theta=0$,
these bounds reduce to
\[
 |\nabla_x k(x,y)|+|\nabla_y k(x,y)|
 \lesssim |x-y|^{-d-1}.
\]
For the prospective quantum $T1$ principle below, we propose the following size condition for the kernel $k$:
\begin{equation}\label{eq:gjp-kernel-size}
    \begin{aligned}
 \bigl\|d_\theta^\rho\bullet
     k
     \bullet d_\theta^{d-\rho}\bigr\|
 \leq C
 \end{aligned}
\end{equation}
for $\rho\in\left\{0,\frac{d}{2},d\right\}.$ At $\theta=0$,
these bounds reduce to
\[
 | k(x,y)|
 \lesssim |x-y|^{-d}.
\]

A formal quantum analogue of the David--Journ\'e theorem would ask
whether
\begin{equation*}\label{eq:formal-quantum-T1}
 \left.
 \begin{gathered}
 T\text{ satisfies the quantum kernel estimates \eqref{eq:gjp-kernel-size} and \eqref{eq:gjp-kernel-smoothness}},\\
 T\text{ satisfies an appropriate quantum weak boundedness condition},\\
 T1,T^*1\in\mathrm{BMO}^{rc}(\R_\theta^d)
 \end{gathered}
 \right\}
 \quad\Longrightarrow\quad
 T\in B(L_2(\R_\theta^d)).
\end{equation*}
This is only a formal scheme rather than a theorem, since the correct
formulations of the first two conditions are themselves part of the
problem. In particular, the classical pointwise kernel conditions
cannot be transferred directly to quantum Euclidean spaces without
first identifying an appropriate noncommutative substitute. We do not know whether there are more natural or intrinsic quantum analogues of the size and smoothness conditions than those in \eqref{eq:gjp-kernel-size} and \eqref{eq:gjp-kernel-smoothness}.
A further difficulty in formulating a complete quantum $T1$
theorem is the absence of a natural quantum analogue of the weak
boundedness property. More precisely, there is no evident intrinsic
notion of spatial support for smooth test functions on quantum
Euclidean spaces, which makes it difficult to formulate the usual
local bump-function testing condition. This leads to the following
question.

\begin{question}
Which concrete quantum kernel estimates, BMO conditions, and weak boundedness conditions characterize
$L_2$-boundedness and reduce to the David--Journ\'e criterion when $\theta=0$? In particular, can the missing conditions be formulated in terms of the boundedness of the associated left and right paraproducts, or alternatively through a natural symmetry assumption?
\end{question}

If suitable quantum kernel estimates and local testing conditions can
be combined with cancellation data controlling the associated
paraproducts, they may lead to a quantum $T1$ theorem, possibly under a
symmetry condition as in Hong, Liu, and Mei~\cite{HongLiuMei2020}.
A possible approach is through transference or the singular integral
framework developed in \cite{GJP}.

\medskip

\noindent\textbf{Acknowledgements.} This work is partially supported by the National Natural Science Foundation of China (Grant Nos. 12325105, W2441002, W2611005).

\noindent\textbf{AI Statement.}
The authors acknowledge the use of ChatGPT for language
polishing, LaTeX editing, and exploratory mathematical discussions during the development and
preparation of this manuscript. Some ideas used in the proof-development process emerged during
interactions with GPT-5.6 Sol. Suggestions from these AI interactions were subsequently examined, reformulated, incorporated into the manuscript, and independently verified by
the authors. The authors assume full responsibility for all mathematical content in the final manuscript.

\end{document}